\documentclass{amsart}

\usepackage{lipsum}
\usepackage{amsfonts,bm,amssymb,amsmath}
\usepackage{graphicx}
\usepackage{epstopdf}
\usepackage[caption=false]{subfig}
\usepackage{pgfplots}
\usepackage{algorithmic}
\usepackage{amsopn}
\usepackage{amsrefs}
\usepackage{float}

\usepackage{hyperref}
\usepackage{cleveref}

\usepackage{color}
\usepackage{diagbox}

\newtheorem{thm}{Theorem}[section]
\newtheorem{definition}{Definition}[section]
\newtheorem{lemma}{Lemma}[section]
\newtheorem{remark}{Remark}[section]

\newtheorem{example}{Example}[section]

\hypersetup{
	colorlinks,
	linkcolor={red!50!black},
	citecolor={blue!50!black},
	urlcolor={blue!80!black}
}

\numberwithin{equation}{section}

\definecolor{newcolor1}{rgb}{.8,.349,.1}
\colorlet{bblue}{blue!50!black}
\crefformat{equation}{(#2#1#3)}
\crefmultiformat{equation}{(#2#1#3)}{ and~(#2#1#3)}{, (#2#1#3)}{ and~(#2#1#3)}

\crefformat{figure}{Figure~#2#1#3}
\crefmultiformat{figure}{Figures~ #2#1#3}{ and~#2#1#3}{, (#2#1#3)}{ and~(#2#1#3)}
\crefformat{table}{Table~#2#1#3}

\def\e{\mbox{\boldmath $e$}}
\def\f{\mbox{\boldmath $f$}}

\def\g{\mbox{\boldmath $g$}}

\def\m{\mbox{\boldmath $m$}}

\def\p{\mbox{\boldmath $p$}}
\def\q{\mbox{\boldmath $q$}}
\def\x{\mbox{\boldmath $x$}}

\def\0{\mbox{\boldmath $0$}}
\def\um{\underline{\bm m}}

\begin{document}

\title[A Semi-implicit Projection Method for LL Equation]{Convergence Analysis of A Second-order Semi-implicit Projection Method for Landau-Lifshitz Equation}

\author{Jingrun Chen}
\address{School of Mathematical Sciences\\ Soochow University\\ Suzhou\\ China.}
\email{jingrunchen@suda.edu.cn}



\author{Cheng Wang}
\address{Mathematics Department\\ University of Massachusetts\\ North Dartmouth\\ MA 02747\\ USA.}
\email{cwang1@umassd.edu}

\author{Changjian Xie}
\address{School of Mathematical Sciences\\ Soochow University\\ Suzhou\\ China.}
\email{20184007005@stu.suda.edu.cn}


\date{\today}



\begin{abstract}
The numerical approximation for the Landau-Lifshitz equation,  the dynamics of magnetization in a ferromagnetic material, is taken
into consideration. This highly nonlinear equation, with a non-convex constraint, has several equivalent forms, and involves solving an auxiliary problem in the infinite domain. All these features have posed interesting challenges in developing numerical methods. In this paper, we first present a fully discrete semi-implicit method  for solving the Landau-Lifshitz equation based on the second-order backward differentiation formula and the one-sided extrapolation (using previous time-step numerical values). A projection step is further used to preserve the length of the magnetization. Subsequently, we provide a rigorous convergence analysis for the fully discrete numerical solution by introducing two sets of approximated solutions to preceed estimation alternatively, with unconditional stability and second-order accuracy in both time and space, provided that the spatial step-size is the same order as the temporal step-size, which remarkably relax restrictions of temporal step-size compared to the implicit schemes. And also, the unique solvability of the numerical solution without any assumptions for the step size in both time and space is theoretically justified, which turns out to be the first such result for the micromagnetics model. All these theoretical properties are verified by numerical examples in both one- and three- dimensional spaces.
\end{abstract}

\maketitle

\section{Introduction}
Micromagnetics is a continuum theory describing magnetization patterns inside ferromagnetic media. The dynamics of magnetization is governed by the Landau-Lifshitz (LL) equation \cite{Landau2015On}. This highly nonlinear equation indicates a non-convex constraint, which has always been a well-known difficulty in the numerical analysis. And also, this equation has several equivalent forms, and an auxiliary problem in the infinite domain has to be involved. All these features have posed interesting challenges in developing numerical methods. In the past several decades, many works have focused on the mathematical theory and numerical analysis of the LL equation \cite{kruzik2006recent, maekawa2006concepts, Shinjo2009Nanomagnetism}. The well-posedness of LL-type equations can be found in~\cite{guo2008landau, Prohl2001Computational, A1985On};
two structures of the solution regularity have been investigated. In the framework of weak solution, the existence of global weak solution in $\mathbb{R}^3$ was proved in \cite{alouges1992global} and in \cite{guo1993landau} on a bounded domain $\Omega\subset \mathbb{R}^2$;  the nonuniqueness of weak solutions was demonstrated in \cite{alouges1992global} as well. In the framework of strong solution, local existence and uniqueness, and global existence and uniqueness with small-energy initial data for strong solutions to the LL equation in $\mathbb{R}^3$ was shown in \cite{carbou2001regular1}. Local existence and uniqueness of strong solutions on a bounded domain $\Omega$ was proved in \cite{carbou2001regular}; global existence and uniqueness of strong solutions for small-energy initial data on a 2-D bounded domain 
was established, provided that $\|\nabla\m_0\|_{H^1(\Omega)}$ is small enough for the bounded domain $\Omega\subset \mathbb{R}^2$. A similar uniqueness analysis was provided in~\cite{melcher2012global} as well. We may refer to \cite{guo1993landau, zhouyulin1991EXISTENCE} for the existence of unique local strong solution. 

Accordingly, numerous numerical approaches have been proposed to demonstrate the mathematical theory; review articles could be found in \cite{cimrak2007survey, kruzik2006recent}. The first finite element work was introduced by Alouges and his collaborators~\cite{alouges2008new, alouges2006convergence, alouges2014convergent, alouges2012convergent}, in which rigorous convergence proof was included with first-order accuracy in time and second-order accuracy in space. This method was further developed to reach almost the second-order temporal accuracy~\cite{alouges2014convergent, kritsikis2014beyond}. In another finite element work by Bartels and Prohl~\cite{bartels2006convergence}, they presented an implicit time integration method with second-order accuracy and unconditional stability. However, a nonlinear solver is needed at each time step, and a theoretical justification of the unique solvability of the numerical solution has not been available. And also, a step-size condition $k=\mathcal{O}(h^2)$ is needed to guarantee the existence of the solution for the fixed point iteration (with $k$ the temporal step-size and $h$ the spatial mesh-size), which is highly restrictive. A similar finite element scheme was reported by Cimr\'{a}k~\cite{cimrak2009convergence}. Again, a nonlinear solver is necessary at each time step, and the same step-size condition has to be imposed. The existing works of finite difference method to the LL equation may be referred to~\cite{weinan2001numerical, fuwa2012finite, jeong2010crank, Kim2017The, Yamada2004Implicit}. In \cite{weinan2001numerical}, a  time stepping method in the form of a projection method was proposed; this method is implicit and unconditionally stable, and the rigorous proof was provided with the first-order accuracy in time and second-order accuracy in space. In \cite{jeong2010crank}, an updated source term was used, and an iteration algorithm was repeatedly performed until the numerical solution converges. In \cite{Kim2017The}, the explicit and implicit mimetic finite difference algorithm was developed.

Regarding to the temporal discretization, the first kind of time-stepping scheme is the Gauss-Seidel projection method proposed by  Wang, Garc\'{i}a-Cervera, and E \cite{wang2001gauss}, in which $|\nabla \m|^2$ was treated as the Lagrange multiplier for the non-convex constraint $|\m|=1$
in the point-wise sense with $\m$ the magnetization vector field. The resulting method is first-order accurate in time and is unconditionally stable. The second kind of time-stepping scheme is called geometric integration method. In \cite{jiang2001hysteresis}, Jiang, Kaper, and Leaf developed the semi-analytic integration method by analytically integrating the system of ODEs, obtained after a spatial discretization of the LL equation. This is an explicit method with first-order accuracy, hence is subject to the CFL constraint. Such an approach has been applied in \cite{krishnaprasad2001cayley} (which yields the same numerical solution as the mid-point method, with second-order accuracy in time), and in a more general setting in \cite{lewis2003geometric} using the Cayley transform to lift the LL equation to the Lie algebra of the three-dimensional rotation group.  In addition, the first, second and fourth-order accurate temporal approximations were examined in~\cite{lewis2003geometric},
which is more amenable for building numerical schemes with the high-order accuracy. The third kind of time-stepping scheme is called the mid-point method~\cite{bertotti2001nonlinear, d2005geometrical}, which is second-order accurate, unconditionally stable, and preserves the Lyapunov and Hamiltonian structures of the LL equation. Moreover, the fourth kind of time-stepping method is the high-order Runge-Kutta algorithms \cite{romeo2008numerical}. Also see other related works~\cite{cimrak2004iterative, di2017linear, jeong2014accurate,  kritsikis2014beyond}, etc.

\par Based on the linearity of the discrete system, we can also classify numerical methods into the explicit scheme~\cite{alouges2006convergence, jiang2001hysteresis}, the fully implicit scheme~\cite{bartels2006convergence, fuwa2012finite, Prohl2001Computational} and the semi-implicit scheme~\cite{cimrak2005error, weinan2001numerical, gao2014optimal, lewis2003geometric, wang2001gauss}. In particular, the semi-discrete schemes are introduced in \cite{Prohl2001Computational} for 2-D and in \cite{cimrak2005error} for 3-D formulation of the LL equation. Error estimates are derived under the existence assumption for the strong solution.

From the perspective of convergence analysis, it is worthy of mentioning \cite{cimrak2004iterative}, in which the fixed point iteration technique was used for handling the nonlinearities; the second-order convergence in time was proved, and was confirmed by numerical examples.
It is noticed that, for all above-mentioned works with the established convergence analysis, a nonlinear solver has to be used at each time step, for the sake of numerical stability. However, the unique solvability analysis for these nonlinear numerical schemes has been a very challenging issue at the theoretical level, due to the highly complicated form in the nonlinear term.  The only relevant analysis was reported in~\cite{fuwa2012finite}, in which the unique solvability was proved under a very restrictive condition, $k \le C h^2$. And also, a projection step has been used in many existing works, to preserve the length of the magnetization. Its nonlinear nature makes a theoretical analysis highly non-trivial. In turn, a derivation of the following numerical scheme is greatly desired: second-order accuracy in time and linearity of the scheme at each time step, so that
the length of magnetization is preserved in the point-wise sense, and an optimal rate error estimate and unconditionally unique solvability analysis could be established at a theoretical level.

\par In this work, we propose and analyze a second-order accurate scheme that satisfies these desired properties. The second-order
backward differentiation formula (BDF) approximation is applied to obtain an intermediate magnetization $\tilde{\m}$, and the right-hand-side nonlinear terms are treated in a semi-implicit style with a second-order extrapolation applied to the explicit coefficients. Such a numerical algorithm leads to a linear system of equations with variable coefficients to solve at each time step. Its unconditionally unique solvability
(no condition is needed for the temporal step-size in terms of spatial step-size) is guaranteed by a careful application of the monotonicity analysis, the so-called Browder-Minty lemma. A projection step is further used to preserve the unit length of magnetization at each time step,
which poses a non-convex constraint. More importantly, we provide a rigorous convergence and error estimate, by the usage of the linearized stability analysis for the numerical error functions. In particular, we notice that, an \textit{a priori} $W_h^{1, \infty}$ bound assumption for the numerical solution at the previous time steps has to be imposed to pass through the convergence analysis. As a consequence, the standard $L^2$ error estimate is insufficient to recover such a bound for the numerical solution. Instead, we have to perform the $H^1$ error estimate, and such a $W_h^{1, \infty}$ bound could be obtained at the next time step as a consequence of the $H^1$ estimate, via the help of the inverse inequality combined with a mild time step-size condition $k=\mathcal{O}(h)$. Careful error estimates for both the original magnetization $\m$ and the intermediate magnetization $\tilde{\m}$ have to be taken into consideration at the projection step (a highly nonlinear operation). To the best of our knowledge, it is the first such result to report an optimal convergence analysis with second order accuracy in both time and space.

\par The rest of this paper is organized as follows. In \cref{sec:main theory}, we introduce the fully discrete numerical scheme and state the main theoretical results: unique solvability analysis and optimal rate convergence analysis. Detailed proofs are also provided in this section. Numerical results are presented in \cref{sec:experiments}, including both the 1-D and 3-D examples to confirm the theoretical analysis.
Conclusions are drawn in \cref{sec:conclusions}.

\section{Main theoretical results}
\label{sec:main theory}
The LL equation reads as
\begin{align}\label{c1}
{\m}_t=-{\m}\times\Delta{\m}-\alpha{\m}\times({\m}\times\Delta{\m})
\end{align}
with
\begin{equation}\label{boundary}
\frac{\partial{\m}}{\partial \boldmath {\nu}}\Big|_{\Gamma}=0,
\end{equation}
where $\Gamma = \partial \Omega$ and $\boldmath {\nu}$ is the unit outward normal vector along $\Gamma$.
Here ${\m}\,:\,\Omega\subset\mathbb{R}^d\to S^2$ represents the magnetization vector field with $|{\m}|=1,\;\forall x\in\Omega$,
$d=1,2,3$ is the spatial dimension, and $\alpha>0$ is the damping parameter. The first term on the right hand side of \cref{c1} is the
gyromagnetic term, and the second term is the damping term. Compared to the original LL equation \cite{Landau2015On}, \cref{c1} only
includes the exchange term which poses the main difficulty in numerical analysis, as done in the literature \cite{weinan2001numerical, cimrak2004iterative, bartels2006convergence, gao2014optimal}. Application of the scheme \cref{eqn:bdf2} to the original
LL equation under external fields will be presented in another publication \cite{Xie2018}. To ease the presentation,
we set $\Omega=[0, 1]$ when $d=1$ and $\Omega=[0, 1]\times[0, 1]\times[0, 1]$ when $d=3$.

\subsection{Finite difference discretization and the fully discrete scheme}\label{discretisations}
The finite difference method is used to approximate \cref{c1} and \cref{boundary}.
Denote the spatial step-szie by $h$ in the 1-D case and divide $[0,1]$ into $N_x$ equal segments;
see the schematic mesh in \cref{spatial_mesh}. Define $x_i=ih$, $i=0,1,2,\cdots,N_x$, with $x_0=0$, $x_{N_x}=1$.
and $\hat{x}_i=x_{i-\frac{1}{2}}=(i-\frac{1}{2})h$, $i=1,\cdots,N_x$.
Denote the magnetization obtained by the numerical scheme at $(\hat{x}_i,t^n)$ by $\m_i^n$.
To approximate the boundary condition \cref{boundary}, we introduce ghost points $x_{-\frac{1}{2}}, x_{N_x+\frac{1}{2}}$
and apply Taylor expansions for $x_{-\frac{1}{2}}$, $x_{\frac{1}{2}}$ at $x_0$, and $x_{N_x+\frac{1}{2}}$, $x_{N_x-\frac{1}{2}}$
at $x_{N_x}$, respectively.  We then obtain a third order extrapolation formula:
\[
\m_1=\m_0, \quad \m_{N_x+1}=\m_{N_x}.
\]
In the 3-D case, we have spatial step-sizes $h_x=\frac{1}{N_x}$, $h_y=\frac{1}{N_y}$, $h_z=\frac{1}{N_z}$ and grid points $(\hat{x}_i,\hat{y}_j,\hat{z}_k)$, with $\hat{x}_i=x_{i-\frac{1}{2}}=(i-\frac{1}{2})h_x$, $\hat{y}_j=y_{j-\frac{1}{2}}=(j-\frac{1}{2})h_y$ and $\hat{z}_k=z_{k-\frac{1}{2}}=(k-\frac{1}{2})h_z$ ( $ 0 \le i \le N_x+1$, $0 \le j \le N_y+1$, $0 \le k \le N_z+1$). The extrapolation formula
along the $z$ direction near $z=0$ and $z=1$ is
\begin{equation}\label{BC-1}
\m_{i,j,1}=\m_{i,j,0}, \quad \m_{i,j,N_z+1}=\m_{i,j,N_z}.
\end{equation}
Extrapolation formulas for the boundary condition along other directions can be derived similarly.

\begin{figure}[htbp]
	\centering
	\begin{tikzpicture}[scale=0.8]
	\draw(1,0)--(14,0);
	\foreach \x/\xtext in {2/$x_0$,4/$x_1$,6/$x_{i-1}$,8/$x_i$,10/$x_{i+1}$,13/$x_{N_x}$}
	\draw(\x,2pt)--(\x,0) node[below] {\xtext};
	\foreach \x/\xtext in {1/$x_{-\frac{1}{2}}$,3/$x_{\frac{1}{2}}$,5/$\cdots$,7/$x_{i-\frac{1}{2}}$,9/$x_{i+\frac{1}{2}}$,11/$\cdots$,12/$x_{N_x-\frac{1}{2}}$,14/$x_{N_x+\frac{1}{2}}$}
	\draw(\x,0)--(\x,0) node[below] {\xtext};
	\draw(1,0) -- node[above=0.4ex] {ghost point}  (2,0);
	\draw(14,0) -- node[above=0.4ex] {ghost point}  (13,0);
	\foreach \x in {1,3,7,9,12,14}
	\draw[blue](\x,0) circle (1.0 mm);
	\end{tikzpicture}
	\caption{Illustration of the 1-D spatial mesh.} \label{spatial_mesh}
\end{figure}
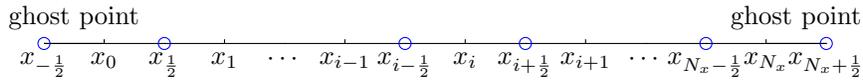

The standard second-order centered difference applied to $\Delta \m$ results in
\begin{align*}
\Delta_h \m_{i,j,k} &=\frac{\m_{i+1,j,k}-2\m_{i,j,k}+\m_{i-1,j,k}}{h_x^2}\\
&+\frac{\m_{i,j+1,k}-2\m_{i,j,k}+\m_{i,j-1,k}}{h_y^2}\nonumber\\
&+\frac{\m_{i,j,k+1}-2\m_{i,j,k}+\m_{i,j,k-1}}{h_z^2},\nonumber
\end{align*}
and the discrete gradient operator $\nabla_h \m$ with $\m=(u, v, w)^T$ reads as
\begin{align*}
\nabla_h\m_{i,j,k} = \begin{bmatrix}
\frac{u_{i+1,j,k}-u_{i,j,k}}{h_x}&\frac{v_{i+1,j,k}-v_{i,j,k}}{h_x}&\frac{w_{i+1,j,k}-w_{i,j,k}}{h_x}\\
\frac{u_{i,j+1,k}-u_{i,j,k}}{h_y}&\frac{v_{i,j+1,k}-v_{i,j,k}}{h_y}&\frac{w_{i,j+1,k}-w_{i,j,k}}{h_y}\\
\frac{u_{i,j,k+1}-u_{i,j,k}}{h_z}&\frac{v_{i,j,k+1}-v_{i,j,k}}{h_z}&\frac{w_{i,j,k+1}-w_{i,j,k}}{h_z}
\end{bmatrix}.
\end{align*}

Denote the temporal step-size by $k$, and define $t^n=nk$, $n\leq \left\lfloor\frac{T}{k}\right\rfloor$
with $T$ the final time. The second-order BDF approximation is applied to the temporal derivative:
\begin{equation*}  \label{BDF2-1}
\frac{\frac{3}{2}\m_h^{n+2}-2\m_h^{n+1}+\frac{1}{2}\m_h^n}{k} = \frac{\partial}{\partial t}\m_h^{n+2} + \mathcal{O}(k^2) . 	 	
\end{equation*}
Note that the right hand side of the above equation is evaluated at $t^{n+2}$, a direct application of the BDF method leads to a fully nonlinear scheme. To overcome this difficulty, we come up with a semi-implicit scheme, in which the nonlinear coefficient is approximated by the
second-order extrapolation formula:
\begin{align}\label{eqn:bdf2}
\frac{\frac{3}{2}{\m}_h^{n+2}-2{\m}_h^{n+1}+\frac{1}{2}{\m}_h^n}{k} &= -\left( 2{\m}_h^{n+1}-{\m}_h^n\right)\times\Delta_h{\m}_h^{n+2} {}\\
&-\alpha\left( 2{\m}_h^{n+1}-{\m}_h^n\right)\times\left((2{\m}_h^{n+1}-{\m}_h^n)\times\Delta_h{\m}_h^{n+2} \right).\nonumber
\end{align}
A projection step is then added to preserve the length of magnetization. This scheme has been used to study domain wall dynamics
under external magnetic fields \cite{Xie2018}. However, this scheme is difficult to conduct the convergence analysis due to
the lack of numerical stability of Lax-Richtmyer type. To overcome this difficulty, we separate the time-marching step and the
projection step in the following way:
\begin{align}\label{scheme-1-1}
\frac{\frac32 \tilde{\m}_h^{n+2} - 2 \tilde{\m}_h^{n+1} + \frac12 \tilde{\m}_h^n}{k}
&=  - \hat{\m}_h^{n+2} \times \Delta_h \tilde{\m}_h^{n+2} \\
&- \alpha \hat{\m}_h^{n+2} \times ( \hat{\m}_h^{n+2} \times \Delta_h \tilde{\m}_h^{n+2}), \nonumber
\end{align}
\begin{align}
\hat{\m}_h^{n+2} &= 2 \m_h^{n+1} - \m_h^n ,\label{cc}\\
\m_h^{n+2} &= \frac{\tilde{\m}_h^{n+2}}{ |\tilde{\m}_h^{n+2}| } . \label{scheme-1-2}
\end{align}

	\begin{remark}
		To kick start the iteration of our method, we are able to obtain the first-order semi-implicit projection scheme using the first-order BDF and the first-order one-sided interpolations in the same manner. The global convergence still maintain the second-order accuracy since the one-step error is first higher order than local truncation error.  
	\end{remark}
	
	\begin{remark}
		To solve the linear system \eqref{scheme-1-1} numerically, we take the sparse LU factorization solver. Afterwards, the solution is projected to the unit sphere at each time step. We thus obtain the numerical solution at final time.  
	\end{remark}

\subsection{Some notations and a few preliminary estimates}

For simplicity of presentation, we assume that $N_x = N_y =N_z=N$ so that $h_x = h_y = h_z =h$. An extension to the general case is straightforward.

First, we introduce the discrete $\ell^2$ inner product and discrete $\| \cdot \|_2$ norm.
\begin{definition}[Inner product and $\| \cdot \|_2$ norm]
	For grid functions $\f_h$ and $\g_h$ over the uniform numerical grid, we define
	\begin{align*}
	\langle {\bm f}_h,{\bm g}_h \rangle = h^d\sum_{\mathcal{I}\in \Lambda_d} \f_{\mathcal{I}}\cdot\g_{\mathcal{I}},
	\end{align*}
	where $\Lambda_d$ is the index set and $\mathcal{I}$ is the index which closely depends on $d$.
	In turn, the discrete $\| \cdot \|_2$ norm is given by
	\begin{equation*}
	\| {\bm f}_h \|_2 =  ( \langle {\bm f}_h,{\bm f}_h \rangle )^{1/2} . \label{defi-L2 norm}
	\end{equation*}
	In addition, the discrete $H_h^1$-norm is given by $\| \f_h \|_{H_h^1}^2 :=\|\f_h\|_2^2+\|\nabla_h \f_h\|_2^2$.
\end{definition}

\begin{definition}[Discrete $\| \cdot \|_\infty$ norm]
	For the grid function $\f_h$ over the uniform numerical grid, we define
	\begin{align*}
	\| \f_h \|_{\infty} = \max_{\mathcal{I}\in\Lambda_d}\|\f_{\mathcal{I}} \|_{\infty} .
	\end{align*}
\end{definition}

\begin{definition}
	For the grid function $\f_h$, we define the average of summation as
	\begin{align*}
	\overline{\f}_h=h^d\sum_{\mathcal{I}\in \Lambda_d}\f_{\mathcal{I}}.
	\end{align*}
\end{definition}

\begin{definition}
	For the grid function $\f_h$ with the normalization condition, due to the Neumann boundary condition imposed (constant functions are in the kernel of $\Delta_h$), we define the discrete $H_h^{-1}$-norm as
	\begin{equation*}
	\|\f_h\|_{-1}^2=\langle(-\Delta_h)^{-1}\f_h,\f_h\rangle.
	\end{equation*}
\end{definition}

The proof of inverse inequality, discrete Gronwall inequality, and summation by parts formula could be obtained in many existing textbooks;
we just cite the results here.
\begin{lemma}(Inverse inequality)\label{ccclemC1}.
	The classical inverse inequality implies that
	\begin{align*}\label{ccc39}
	\|{\e}_h^{n}\|_{\infty} \leq {h}^{-d/2}\|{\e}_h^{n}\|_2 ,  \quad
	\|\nabla_h{\e}_h^{n}\|_{\infty} \leq {h}^{-d/2}\|\nabla_h{\e}_h^{n}\|_2.\nonumber
	\end{align*}
\end{lemma}

\begin{lemma}(Discrete Gronwall inequality)\label{ccclem1}. Let $\{\alpha_j\}_{j\geq 0}$, $\{\beta_j\}_{j\geq 0}$ and $\{\omega_j\}_{j\geq 0}$ be sequences of real numbers such that
	\begin{equation*}
	\alpha_j\leq \alpha_{j+1},\quad \beta_j\geq 0,\quad and \quad \omega_j\leq \alpha_j+\sum_{i=0}^{j-1}\beta_i\omega_i , \quad \forall j \geq 0.
	\end{equation*}
	Then it holds that
	\begin{equation*}
	\omega_j\leq \alpha_j\exp\left\{\sum_{i=0}^{j-1}\beta_i\right\} , \quad \forall j \geq 0.
	\end{equation*}
\end{lemma}

\begin{lemma}[Summation by parts]\label{summation}
	For any grid functions $\f_h$ and $\g_h$, with $f_h$ satisfying the discrete boundary condition~\cref{BC-1}, the following identity is valid:
	\begin{align}\label{sum1}
	\left\langle -\Delta_h \f_h,\g_h\right\rangle = \left\langle \nabla_h \f_h,\nabla_h \g_h\right\rangle .
	\end{align}
\end{lemma}

The following estimate will be utilized in the convergence analysis. In the sequel, for simplicity of our notation, we will use the uniform constant $\mathcal{C}$ to denote all the controllable constants in this paper.

\begin{lemma}[Discrete gradient acting on cross product]\label{lem27}
	For grid functions $\f_h$ and $\g_h$ over the uniform numerical grid, we have
	\begin{align}
	\|\nabla_h({\f}\times{\g})_h\|_2^2 &\leq \mathcal{C}\Big(\|{\f}_h\|_{\infty}^2\cdot \|\nabla_h {\g}_h\|_2^2+\|{\g}_h\|_{\infty}^2\cdot \|\nabla_h{\f}_h\|_2^2\Big),\label{lem27_1}\\
	\left\langle (\f_h\times \Delta_h \g_h)\times \f_h,\g_h \right\rangle &=\left\langle \f_h\times(\g_h\times \f_h),\Delta_h\g_h\right\rangle\label{lem27_2},\\
	\left\langle\f_h\times(\f_h\times \g_h),\g_h \right\rangle &=-\|\f_h\times \g_h\|_2^2\label{lem27_3}.
	\end{align}
\end{lemma}

\begin{proof} Without loss of generality, we only look at the 1-D case; an extension to the 3-D case is straightforward.	We begin with the following expansion
	\begin{align}  \label{expansion-1}
	\left[\nabla_h({\f}\times {\g})\right]_{i+\frac{1}{2}} &=\frac{\f_{i+1}\times \g_{i+1}-\f_i\times \g_i}{h}  \\
	&=\frac{\f_{i+1}-\f_i}{h}\times \g_{i+1}+\f_i \times \frac{\g_{i+1}-\g_i}{h}\nonumber\\
	&=\left(\nabla_h \f\right)_{i+\frac{1}{2}}\times \g_{i+1}+\f_i\times \left(\nabla_h \g\right)_{i+\frac{1}{2}}.\nonumber
	\end{align}
	In turn, an application of the discrete H\"older inequality to~\cref{expansion-1} yields~ \cref{lem27_1}. Also note that
	\begin{align*}
	\left\langle (\f_h\times \Delta_h \g_h)\times \f_h,\g_h \right\rangle &=-\left\langle\g_h\times \f_h,\f_h\times \Delta_h\g_h \right\rangle\\
	&=\left\langle \f_h\times(\g_h\times \f_h),\Delta_h\g_h\right\rangle,\nonumber
	\end{align*}
	and
	\begin{align*}
	\left\langle\f_h\times(\f_h\times \g_h),\g_h \right\rangle
	&=	\left\langle \f_h\times\g_h,\g_h\times \f_h \right\rangle\\
	&=-\|\f_h\times \g_h\|_2^2 . \nonumber
	\end{align*}
\end{proof}

The following estimate will be used in the error estimate at the projection step.

\begin{lemma}  \label{lem 6-0}
	Consider $\um_h = \m_e + h^2 \m^{(1)}$ with $\m_e\in W^{1, \infty}$ the exact solution to \cref{c1} and $| \m_e | = 1$ at a point-wise level, and $\| \m^{(1)} \|_\infty + \| \nabla_h \m^{(1)} \|_\infty \le \mathcal{C}$. For any numerical solution $\tilde{\m}_h$, we define $\m_h = \frac{\tilde{\m}_h}{ | \tilde{\m}_h |}$. Suppose both numerical profiles satisfy the following $W_h^{1, \infty}$ bounds
	\begin{align}
	& |\tilde{\m}_h| \ge \frac12 ,  \quad \mbox{at a point-wise level} ,  \label{lem 6-1-0}
	\\
	& \| \m_h \|_{\infty} + \| \nabla_h \m_h \|_\infty \le M ,  \quad \| \tilde{\m}_h \|_{\infty} + \| \nabla_h \tilde{\m}_h \|_\infty \le M,  \label{lem 6-1}
	\end{align}
	and we denote the numerical error functions as $\e_h = \m_h - \um_h$, $\tilde{\e}_h = \tilde{\m}_h - \um_h$. Then the following estimate is valid
	\begin{equation}
	\| \e_h \|_2 \le 2 \| \tilde{\e}_h \|_2 + \mathcal{O} (h^2) ,  \quad
	\| \nabla_h \e_h \|_2 \le \mathcal{C} ( \| \nabla_h \tilde{\e}_h \|_2
	+ \| \tilde{\e}_h \|_2 ) + \mathcal{O} (h^2) . \label{lem 6-2}
	\end{equation}
\end{lemma}

\begin{proof}
	A direct calculation shows that
	\begin{align}
	{\e}_h &= {\m}_h - \um_h=\frac{\tilde{\m}_h}{|\tilde{\m}_h|}-{\um}_h
	= \tilde{\m}_h - {\um}_h +\frac{\tilde{\m}_h}{|\tilde{\m}_h|} - \tilde{\m}_h \nonumber\\
	&= \tilde{\e}_h +\frac{\tilde{\m}_h}{|\tilde{\m}_h|}(|{\um}_h|-|\tilde{\m}_h|) + \frac{\tilde{\m}_h}{ |\tilde{\m}_h|} ( 1 - |{\um}_h|).  \label{lem 6-3}
	\end{align}
	Since $\Big||{\um}_h|-|\tilde{\m}_h|\Big| \le |{\um}_h-\tilde{\m}_h|$, we get
	\begin{equation} \label{lem 6-4} 	
	\left\|\tilde{\e}_h +\frac{\tilde{\m}_h}{|\tilde{\m}_h|}(|{\um}_h|-|\tilde{\m}_h|)\right\|_2\leq \|\tilde{\e}_h \|_2+\|\tilde{\e}_h \|_2=2\|\tilde{\e}_h\|_2 .
	\end{equation}
	For the last term on the right hand side of~\cref{lem 6-3}, we observe that
	\begin{align}\label{lem 6-5-1}
	\left| 1 - |{\um}_h | \right| = \left| | \m_e | - |{\um}_h | \right|
	\le \left|  \m_e - {\um}_h \right| = h^2 | \m^{(1)} | = \mathcal{O} (h^2) ,
	\end{align}
	
	which in turn yields
	\begin{align} \label{lem 6-5-2}
	\left\| \frac{\tilde{\m}_h}{ |\tilde{\m}_h|} ( 1 - |{\um}_h|) \right\|_2 = \mathcal{O} (h^2).  
	\end{align}
	
	As a result, a substitution of~\cref{lem 6-4} and~\cref{lem 6-5-2} into~\cref{lem 6-3} leads to the first estimate in~\cref{lem 6-2}.

	For the second inequality, we notice that
	\begin{align}\label{lem 6-6}
	\nabla_h{\e}_h &= \nabla_h\frac{\tilde{\m}_h}{|\tilde{\m}_h|} - \nabla_h {\um}_h = \nabla_h\left[ \frac{\tilde{\m}_h}{|\tilde{\m}_h|}-\frac{{\um}_h}{|\tilde{\m}_h|}\right]
	+\nabla_h\left[\frac{{\um}_h}{|\tilde{\m}_h|}-{\um}_h \right] \\
	&= \nabla_h \frac{\tilde{\e}_h}{|\tilde{\m}_h|}+\nabla_h\left[ \frac{{\um}_h}{|\tilde{\m}_h|}(1-|\tilde{\m}_h|)\right] . \nonumber
	\end{align}	
	The analysis for the first part is straightforward:
	\begin{align}\label{lem 6-7}
	\left\|\nabla_h \frac{\tilde{\e}_h}{|\tilde{\m}_h^n|}\right\|_2 &\leq \left\|\frac{1}{\tilde{\m}_h}\right\|_{\infty} \cdot \|\nabla_h\tilde{\e}_h\|_2
	+ \|\tilde{\e}_h \|_2 \cdot \left\|\nabla_h\frac{1}{|\tilde{\m}_h|}\right\|_{\infty}
	\le \mathcal{C}\|\nabla_h\tilde{\e}_h\|_2+\mathcal{C}\|\tilde{\e}_h\|_2 .
	\end{align}
	For the second part, we rewrite it as
	\begin{align*}\label{lem 6-8-1}
	& \frac{{\um}_h}{|\tilde{\m}_h|}(1-|\tilde{\m}_h|)
	= \frac{{\um}_h}{|\tilde{\m}_h|}\frac{({\um}_h+\tilde{\m}_h)({\um}_h-\tilde{\m}_h)}{1+|\tilde{\m}_h|}
	+ \frac{{\um}_h}{|\tilde{\m}_h|}\frac{(\m_e + {\um}_h)( \m_e - {\um}_h)}{1+|\tilde{\m}_h|} , 	
	\end{align*}	
	based on the fact $| \m_e | \equiv 1$. In turn, the following two bounds could be derived:
	\begin{align*}
	& 
	\left\| \nabla_h \left[ \frac{{\um}_h}{|\tilde{\m}_h|} \frac{({\um}_h+\tilde{\m}_h)({\um}_h-\tilde{\m}_h)}{1+|\tilde{\m}_h|}\right]\right\|_2\\
	\le \, & \left\|\frac{{\um}_h}{|\tilde{\m}_h|}\frac{{\um}_h+\tilde{\m}_h}{1+|\tilde{\m}_h|}\right\|_{\infty} \cdot \left\|\nabla_h({\um}_h-\tilde{\m}_h)\right\|_2 \nonumber \\
	& +\|{\um}_h-\tilde{\m}_h\|_2 \cdot \left\|\nabla_h\left[\frac{{\um}_h}{|\tilde{\m}_h|}\frac{{\um}_h+\tilde{\m}_h}{1+|\tilde{\m}_h|} \right]\right\|_{\infty}\nonumber\\
	\le \, & \mathcal{C}\|\nabla_h\tilde{\e}_h \|_2+\mathcal{C}\|\tilde{\e}_h\|_2 , \nonumber
	\end{align*}
	and	
	\begin{align*} 
	&
	\left\| \nabla_h \left[ \frac{{\um}_h}{|\tilde{\m}_h|} \frac{( \m_e + {\um}_h)( \m_e - {\um}_h)}{1+|\tilde{\m}_h|}\right]\right\|_2\\
	\le \, & \left\|\frac{{\um}_h}{|\tilde{\m}_h|}\frac{ \m_e + {\um}_h}{1+|\tilde{\m}_h|}\right\|_{\infty} \cdot \left\|\nabla_h( \m_e - {\um}_h )\right\|_2 \nonumber \\
	& +\| \m_e - {\um}_h \|_2 \cdot \left\|\nabla_h\left[\frac{{\um}_h}{|\tilde{\m}_h|}\frac{ \m_e + {\um}_h }{1+|\tilde{\m}_h|} \right]\right\|_{\infty} = \mathcal{O} (h^2) . \nonumber
	\end{align*}
	Therefore, we obtain
	\begin{equation} \label{lem 6-8-4}
	\left\|\nabla_h\frac{{\um}_h^n}{|\tilde{\m}_h^n|}(1-|\tilde{\m}_h^n|)\right\|_2
	\le  \mathcal{C} ( \|\nabla_h\tilde{\e}_h \|_2 + \|\tilde{\e}_h\|_2 ) +  \mathcal{O} (h^2) .
	\end{equation} 		
	Finally, a substitution of~\cref{lem 6-7} and \cref{lem 6-8-4} into~\cref{lem 6-6} yields the second inequality in~\cref{lem 6-2}. This completes the proof of~\cref{lem 6-0}.
\end{proof}

\subsection{The main theoretical results}\label{leading}
\par The first theoretical result is the unique solvability analysis of scheme~\cref{scheme-1-1}-\cref{scheme-1-2}.  We observe that the unique solvability for \cref{scheme-1-1} could be simplified as the analysis for
\begin{equation}\label{scheme-alt-1}
\frac{\frac32 \tilde{\m}_h - \p_h}{k}
=  - \hat{\m}_h \times \Delta_h \tilde{\m}_h
- \alpha \hat{\m}_h \times ( \hat{\m}_h \times \Delta_h \tilde{\m}_h)
\end{equation}
with $\p_h$, $\hat{\m}_h$ given.

\begin{thm} \label{thm:solvability}
	Given $\p_h$, $\hat{\m}_h$, the numerical scheme \cref{scheme-alt-1} is uniquely solvable.
\end{thm}

\par To facilitate the unique solvability analysis for \cref{scheme-alt-1}, we denote $\q_h = - \Delta_h \tilde{\m}_h$. Note that $\overline{\q_h} =0$, due to the Neumann boundary condition for $\tilde{\m}_h$. Meanwhile, we observe that $\tilde{\m}_h \ne (-\Delta_h)^{-1} \q_h$ in general, since $\overline{\tilde{\m_h}} \ne 0$. Instead, $\tilde{\m}_h$ could be represented as follows:
\begin{equation*}
\tilde{\m}_h = (-\Delta_h)^{-1} \q_h + C_{\q_h}^* \quad \mbox{with} \, \, \,
C_{\q_h}^* = \frac23 \Bigl( \overline{\p_h} + k  \overline{ \hat{\m}_h \times \q_h}
+ \alpha k \overline{\hat{\m}_h \times ( \hat{\m}_h \times \q_h) } \Bigr)
\label{m-representation-1}
\end{equation*}
and $\hat{\m}_h$ given by \cref{cc}. \cref{scheme-alt-1} is then rewritten as
\begin{equation}
G (\q_h) := \frac{\frac32 ( (-\Delta_h)^{-1} \q_h + C_{\q_h}^* ) - \p_h}{k}
- \hat{\m}_h \times \q_h
- \alpha \hat{\m}_h \times ( \hat{\m}_h \times \q_h) = \0.
\label{scheme-alt-2}
\end{equation}

\begin{lemma}[Browder-Minty lemma \cite{browder1963nonlinear,minty1963monotonicity}] \label{lem:Browder}
	Let X be a real, reflexive Banach space and let $T: X \to X'$ (the dual space of $X$) be bounded, continuous,
	coercive (i.e., $\frac{ (T (u) , u ) }{ \| u \|_X } \to +\infty$, as $\| u \|_X \to +\infty$) and monotone.
	Then for any $g \in X'$ there exists a solution $u \in X$ of the equation $T (u) =g$.
	
	Furthermore, if the operator $T$ is strictly monotone, then the solution $u$ is unique.
\end{lemma}

Then we proceed into the proof of \cref{thm:solvability}.

\begin{proof}
	Recall that~\cref{scheme-alt-1} is equivalent to~\cref{scheme-alt-2}.
	For any $\q_{1,h}$, $\q_{2,h}$ with $\overline{\q_{1,h}} = \overline{\q_{2,h}}=0$,
	we denote $\tilde{\q}_h = \q_{1,h} - \q_{2,h}$ and derive the following monotonicity estimate:
	\begin{align*}\label{solvability-1}
	&\langle G (\q_{1,h}) - G (\q_{2,h}) , \q_{1,h} - \q_{2,h} \rangle \\
	&=\frac{3}{2 k} \Bigl(
	\langle (-\Delta_h)^{-1} \tilde{\q}_h , \tilde{\q}_h \rangle
	+ \langle C_{\q_{1,h}}^* - C_{\q_{1,h}}^* , \tilde{\q}_h \rangle \Bigr)\nonumber\\
	& - \langle \hat{\m}_h \times \tilde{\q}_h , \tilde{\q}_h \rangle
	- \alpha \langle\hat{\m}_h \times ( \hat{\m}_h \times \tilde{\q}_h ) ,
	\tilde{\q}_h \rangle\nonumber\\
	&\ge \frac{3}{2 k} \Bigl(
	\langle (-\Delta_h)^{-1} \tilde{\q}_h , \tilde{\q}_h \rangle
	+ \langle C_{\q_{1,h}}^* - C_{\q_{2,h}}^* , \tilde{\q}_h \rangle \Bigr) \nonumber\\
	&= \frac{3}{2 k}
	\langle (-\Delta_h)^{-1} \tilde{\q}_h , \tilde{\q}_h \rangle
	= \frac{3}{2 k} \| \tilde{\q}_h \|_{-1}^2 \ge 0 . \nonumber
	\end{align*}
	Note that the following equality and inequality have been applied in the second step:
	\begin{align*} 
	\langle \hat{\m}_h \times \tilde{\q}_h , \tilde{\q}_h \rangle  & = 0 , \quad
	\langle \hat{\m}_h \times ( \hat{\m}_h \times \tilde{\q}_h ), \tilde{\q}_h \rangle \le 0 .
	\end{align*}
	The third step is based on the fact that both $C_{\q_{1,h}}^*$ and $C_{\q_{2,h}}^*$ are constants, and $\overline{\q_{1,h}} = \overline{\q_{2,h}}=0$, so that $\langle C_{\q_{1,h}}^* - C_{\q_{2,h}}^* , \tilde{\q}_h \rangle = 0$.
	
	Moreover, for any $\q_{1,h}$, $\q_{2,h}$ with $\overline{\q_{1,h}} = \overline{\q_{2,h}}=0$, we get
	\begin{eqnarray*}
		\langle G (\q_{1,h}) - G (\q_{2,h}) , \q_{1,h} - \q_{2,h} \rangle
		\ge \frac{3}{2 k} \| \tilde{\q}_h \|_{-1}^2 > 0 ,  \quad
		\mbox{if $\q_{1,h} \ne \q_{2,h}$} , \label{solvability-3}
	\end{eqnarray*}
	and the equality only holds when $\q_{1,h} = \q_{2,h}$.
	
	Therefore, an application of \cref{lem:Browder} implies a unique solution of both \cref{scheme-alt-2} and  \cref{scheme-alt-1},
	which completes the proof of \cref{thm:solvability}.
\end{proof}

The second theoretical result is the optimal rate convergence analysis. 

\begin{thm}\label{cccthm2} Let $\m_e \in C^3 ([0,T]; C^0) \cap L^{\infty}([0,T]; C^4)$ be a smooth solution of \cref{c1} with the initial data $\m_e ({\x},0)=\m_e ^0({\x})$ and ${\m}_h$ be the numerical solution of the equation~\cref{scheme-1-1}-\cref{scheme-1-2} with the initial data ${\m}_h^0=\m_{e,_h}^0$ and $\m_h^1= \m _{e,h}^1$. Suppose that the initial error satisfies $\|\m_{e,h}^\ell - \m_h^\ell \|_2 +\|\nabla_h ( \m_{e,h}^\ell - \m_h^\ell ) \|_2 = \mathcal{O} (k^2 + h^2),\,\ell=0,1$, and $k\leq \mathcal{C}h$. Then the following convergence result holds as $h$ and $k$ goes to zero:
	\begin{align}
	\| \m_{e,h}^n - \m_h^n \|_{2}+ \|\nabla_h ( \m_{e,h}^n - \m_h^n ) \|_{2} &\leq \mathcal{C}(k^2+h^2) , \quad \forall n \ge 2 ,
	\end{align}	
	in which the constant $\mathcal{C}>0$ is independent of $k$ and $h$.
\end{thm}

\begin{proof}
	First, we construct an approximate solution $\underline{\m}$:
	\begin{equation}
	\underline{\m} = \m_e + h^2 \m^{(1)} ,  \label{exact-1}
	\end{equation}
	in which the auxiliary field $\m^{(1)}$ satisfies the following Poisson equation
	\begin{align}  \label{exact-2}
	& \Delta \m^{(1)} = \hat{C}  \quad \mbox{with} \, \, \,
	\hat{C} = \frac{1}{| \Omega|} \int_{\partial \Omega} \,
	\partial_{\boldmath \nu}^3 \m_e \, \textrm{d} s ,  \\
	& \partial_z \m^{(1)} \mid_{z=0} = - \frac{1}{24} \partial_z^3 \m_e \mid_{z=0} ,  \quad
	\partial_z \m^{(1)} \mid_{z=1} = \frac{1}{24} \partial_z^3 \m_e \mid_{z=1} ,  \nonumber
	\end{align}
	with boundary conditions along $x$ and $y$ directions defined in a similar way.
	
	The purpose of such a construction will be illustrated later. Then we extend the approximate profile $\underline{\m}$ to the numerical ``ghost" points, according to the extrapolation formula~\cref{BC-1}:
	\begin{equation}
	\underline{\m}_{i,j,0}= \underline{\m}_{i,j,1} , \quad
	\underline{\m}_{i,j,N_z+1} = \underline{\m}_{i,j,N_z} ,  \label{exact-3}
	\end{equation}
	and the extrapolation for other boundaries can be formulated in the same manner. Subsequently, we prove that such an extrapolation yields a higher order $\mathcal{O}(h^5)$ approximation, instead of the standard $\mathcal{O}(h^3)$ accuracy. Also see the related works~\cite{STWW2003, Wang2000, Wang2004} in the existing literature.
	
	Performing a careful Taylor expansion for the exact solution around the boundary section $z=0$, combined with the mesh point values: $\hat{z}_0 = - \frac12 h$, $\hat{z}_1 = \frac12 h$, we get
	\begin{align}
	\m_e  (\hat{x}_i, \hat{y}_j, \hat{z}_0 )
	&= \m_e (\hat{x}_i, \hat{y}_j, \hat{z}_1 )
	- h \partial_z \m_e (\hat{x}_i, \hat{y}_j, 0 )
	- \frac{h^3}{24} \partial_z^3 \m_e (\hat{x}_i, \hat{y}_j, 0 )
	+   \mathcal{O}(h^5) \nonumber
	\\
	&= \m_e (\hat{x}_i, \hat{y}_j, \hat{z}_1 )
	- \frac{h^3}{24} \partial_z^3 \m_e (\hat{x}_i, \hat{y}_j, 0 )
	+   \mathcal{O}(h^5) ,  \label{exact-4}
	\end{align}
	in which the homogenous boundary condition has been applied in the second step. A similar Taylor expansion for the constructed profile $\m^{(1)}$ reveals that
	\begin{align}
	\m^{(1)}  (\hat{x}_i, \hat{y}_j, \hat{z}_0 )
	&= \m^{(1)} (\hat{x}_i, \hat{y}_j, \hat{z}_1 )
	- h \partial_z \m^{(1)} (\hat{x}_i, \hat{y}_j, 0 )
	+   \mathcal{O}(h^3) \nonumber
	\\
	&= \m^{(1)} (\hat{x}_i, \hat{y}_j, \hat{z}_1 )
	+ \frac{h}{24} \partial_z^3 \m_e (\hat{x}_i, \hat{y}_j, 0 )
	+   \mathcal{O}(h^3)  \label{exact-5}
	\end{align}
	with the boundary condition in~\cref{exact-2} applied. In turn, a substitution of~\cref{exact-4}-\cref{exact-5} into~\cref{exact-1} indicates that
	\begin{equation}
	\underline{\m}  (\hat{x}_i, \hat{y}_j, \hat{z}_0 )
	= \underline{\m} (\hat{x}_i, \hat{y}_j, \hat{z}_1 )
	+   \mathcal{O}(h^5) .   \label{exact-6}
	\end{equation}
	In other words, the extrapolation formula~\cref{exact-3} is indeed $\mathcal{O}(h^5)$ accurate.
	
	As a result of the boundary extrapolation estimate~\cref{exact-6}, we see that the discrete Laplacian of $\underline{\m}$ yields the second-order accuracy, even at the mesh points around the boundary sections:
	\begin{equation}
	\Delta_h \underline{\m}_{i,j,k} = \Delta \m_e (\hat{x}_i, \hat{y}_j, \hat{z}_k )
	+   \mathcal{O}(h^2) , \quad \forall 1 \le i,j,k \le N .  \label{consistency-1}
	\end{equation}
	Moreover, a detailed calculation of Taylor expansion, in both time and space, leads to the following truncation error estimate:
	\begin{align} \label{consistency-2}
	\frac{\frac{3}{2} \underline{\m}_h^{n+2} - 2\underline{\m}_h^{n+1}
		+ \frac{1}{2} \underline{\m}_h^n}{k} &= -\left( 2 \underline{\m}_h^{n+1} - \underline{\m}_h^n\right) \times \Delta_h \underline{\m}_h^{n+2} + \tau^{n+2} \\
	&-\alpha\left( 2 \underline{\m}_h^{n+1} - \underline{\m}_h^n\right) \times \left(
	( 2 \underline{\m}_h^{n+1} - \underline{\m}_h^n) \times \Delta_h \underline{\m}_h^{n+2} \right) , \nonumber
	\end{align}
	with $\| \tau^{n+2} \|_2 \le \mathcal{C} (k^2+h^2)$. Meanwhile, we introduce the numerical error functions $\tilde{\e}_h^n=\underline{\m}_h^n-\tilde{\m}_h^n$, ${\e}_h^n=\underline{\m}_h^n-\m_h^n$, at a point-wise level.
	In other words, instead of a direct comparison between the numerical solution and the exact solution,
	we analyze the error function between the numerical solution and the constructed solution ${\um}_h$, due to its higher order consistency estimate~\cref{exact-6} around the boundary. A subtraction of \cref{scheme-1-1}-\cref{scheme-1-2} from the consistency estimate~\cref{consistency-2} leads to
	the error function evolution system:
	\begin{align}\label{ccc73}
	\frac{\frac{3}{2}\tilde{\e}_h^{n+2}-2\tilde{\e}_h^{n+1}+\frac{1}{2}\tilde{\e}_h^n}{k} &= -\left( 2{\m}_h^{n+1}-{\m}_h^n\right)\times\Delta_h\tilde{\e}_h^{n+2}-\left( 2{\e}_h^{n+1}-{\e}_h^n\right)\times\Delta_h {\um}_h^{n+2}\\
	&- \alpha\left( 2{\m}_h^{n+1}-{\m}_h^n\right)\times\left( (2{\m}_h^{n+1}-{\m}_h^n)\times\Delta_h\tilde{\e}_h^{n+2}\right)\nonumber\\
	&- \alpha\left( 2{\m}_h^{n+1}-{\m}_h^n\right)\times\left( (2{\e}_h^{n+1}-{\e}_h^n)\times\Delta_h {\um}_h^{n+2} \right)\nonumber\\
	&- \alpha\left( 2{\e}_h^{n+1}-{\e}_h^n\right)\times\left( (2{\um}_h^{n+1}-{\um}_h^n)\times\Delta_h {\um}_h^{n+2}\right)+\tau^{n+2} . \nonumber
	\end{align}
	
	Before we proceed into the formal error estimate, we establish the bound for the constructed approximate solution ${\um}$ and the numerical solution $\m_h$. For the approximate profile $\um \in L^{\infty}([0,T],C^4)$, which turns out to be the exact solution and an $\mathcal{O} (h^2)$ correction term, we still use $\mathcal{C}$ to denote its bound:
	\begin{align}
	\|\nabla_h^r\um\|_{\infty}\leq \mathcal{C}, \quad r=0,1,2,3.  \label{bound-1}
	\end{align}
	In addition, we make the following \textit{a priori} assumption for the numerical error function:
	\begin{equation} \label{bound-2}
	\|{\e}_h^k \|_{\infty}+\|\nabla_h{\e}_h^k \|_{\infty} \le \frac13 ,  \, \, \,
	\| \tilde{\e}_h^k \|_{\infty}+\|\nabla_h \tilde{\e}_h^k \|_{\infty} \le \frac13 ,
	\quad \mbox{for $k= \ell, \ell+1$} .
	\end{equation}
	Such an assumption will be recovered by the convergence analysis at time step $t^{\ell+2}$. In turn, an application of triangle inequality yields the desired $W_h^{1,\infty}$ bound for the numerical solutions $\m_h$ and $\tilde{\m}_h$:
	\begin{align}
	\|{\m}_h^k \|_{\infty} &= \|\um_h^k -{\e}_h^k \|_{\infty}\leq\|\um_h^k \|_{\infty}+\|{\e}_h^k \|_{\infty}\leq \mathcal{C}+ \frac13 ,  \label{bound-3} \\
	\|\nabla_h{\m}_h^k \|_{\infty} &= \|\nabla_h\um_h^k -\nabla_h{\e}_h^k \|_{\infty}\leq\|\nabla_h\um_h^k \|_{\infty}+\|\nabla_h{\e}_h^k \|_{\infty}\leq \mathcal{C}+\frac13 , \nonumber \\
	\| \tilde{\m}_h^k \|_{\infty} & \leq \mathcal{C}+ \frac13 , \quad
	\| \nabla_h \tilde{\m}_h^k \|_{\infty} \leq \mathcal{C}+ \frac13
	\quad \mbox{(similar derivation)} . \label{bound-4}  			
	\end{align}
	
	Then we perform a discrete $L^2$ error estimate at $t^{\ell+2}$ using the mathematical induction. 
	By taking a discrete inner product with the numerical error equation \cref{ccc73} by $\tilde{\e}_h^{\ell+2}$ gives that
	\begin{align}\label{rhs}
	R.H.S.&= \left\langle-\left( 2{\m}_h^{\ell+1}-{\m}_h^\ell\right)\times\Delta_h\tilde{\e}_h^{\ell+2},\tilde{\e}_h^{\ell+2} \right\rangle\\
	&-\left\langle\left( 2{\e}_h^{\ell+1}-{\e}_h^\ell\right)\times\Delta_h\underline{\m}_h^{\ell+2}, \tilde{\e}_h^{\ell+2}\right\rangle +\left\langle\tau^{\ell+2},\tilde{\e}_h^{\ell+2} \right\rangle \nonumber\\
	&-\alpha \left\langle \left( 2{\m}_h^{\ell+1}-{\m}_h^\ell\right)\times\left( (2{\m}_h^{\ell+1}-{\m}_h^\ell)\times\Delta_h\tilde{\e}_h^{\ell+2}\right),\tilde{\e}_h^{\ell+2}\right\rangle\nonumber\\
	&-\alpha \left\langle\left( 2{\m}_h^{\ell+1}-{\m}_h^\ell\right)\times\left( (2{\e}_h^{\ell+1}-{\e}_h^\ell)\times\Delta_h\underline{\m}_h^{\ell+2} \right),\tilde{\e}_h^{\ell+2} \right\rangle\nonumber\\
	&-\alpha \left\langle \left( 2{\e}_h^{\ell+1}-{\e}_h^\ell\right)\times\left( (2\underline{\m}_h^{\ell+1}-\underline{\m}_h^\ell)\times\Delta_h\underline{\m}_h^{\ell+2}\right),\tilde{\e}_h^{\ell+2}\right\rangle\nonumber\\
	&=: \tilde{I}_1+\tilde{I}_2+\tilde{I}_3+\tilde{I}_4+\tilde{I}_5+\tilde{I}_6.\nonumber
	\end{align}
	
	\begin{itemize}
		\item Estimate of $\tilde{I}_1$: A combination of the summation by parts formula~\cref{sum1} (notice that the numerical error function $\tilde{\e}$ satisfies the homogeneous Neumann boundary condition~\cref{BC-1}) and inequality~\cref{lem27_1} results in 	
		\begin{align}\label{I1}
		\tilde{I}_1 = \, &\left\langle-\left( 2{\m}_h^{\ell+1}-{\m}_h^\ell\right)\times\Delta_h\tilde{\e}_h^{\ell+2},\tilde{\e}_h^{\ell+2} \right\rangle\\
		= \, & \left\langle \tilde{\e}_h^{\ell+2}\times\left( 2{\m}_h^{\ell+1}-{\m}_h^\ell\right), -\Delta_h\tilde{\e}_h^{\ell+2}\right\rangle\nonumber\\
		= \, & \left\langle \nabla_h\Big[\tilde{\e}_h^{\ell+2}\times\left( 2{\m}_h^{\ell+1}-{\m}_h^\ell\right)\Big] , \nabla_h\tilde{\e}_h^{\ell+2}\right\rangle \nonumber\\
		\le \, & \mathcal{C}\Big(\|\nabla_h\tilde{\e}_h^{\ell+2}\|_2^2+\|\nabla_h\tilde{\e}_h^{\ell+2}\|_2^2\cdot \|2{\m}_h^{\ell+1}-{\m}_h^\ell\|_{\infty}^2 \nonumber \\
		& +\|\tilde{\e}_h^{\ell+2}\|_2^2\cdot \|\nabla_h(2{\m}_h^{\ell+1}-{\m}_h^\ell)\|_{\infty}^2\Big) \nonumber\\
		\le \, & \mathcal{C} ( \|\nabla_h\tilde{\e}_h^{\ell+2}\|_2^2  + \|\tilde{\e}_h^{\ell+2}\|_2^2 ) . \nonumber
		\end{align}
		\item Estimate of $\tilde{I}_2$:
		\begin{align}\label{I2}
		\tilde{I}_2 =\, & -\left\langle\left( 2{\e}_h^{\ell+1}-{\e}_h^\ell\right)\times\Delta_h\um_h^{\ell+2}, \tilde{\e}_h^{\ell+2}\right\rangle \\
		\le \, & \frac{1}{2}\big[\|\tilde{\e}_h^{\ell+2}\|_2^2+\|2{\e}_h^{\ell+1}-{\e}_h^\ell\|_2^2\cdot \|\Delta_h\um_h^{\ell+2}\|_{\infty}^2 \big]\nonumber\\
		\le \, & \mathcal{C} ( \|\tilde{\e}_h^{\ell+2}\|_2^2 + \| \e_h^{\ell+1}\|_2^2
		+ \| \e_h^{\ell}\|_2^2 ) . \nonumber
		\end{align}
		
		\item Estimate of the truncation error term $\tilde{I}_3$: An application of Cauchy inequality gives
		\begin{align}\label{I3}
		\tilde{I}_3 = \left\langle\tau^{\ell+2},\tilde{\e}_h^{\ell+2} \right\rangle \leq \mathcal{C}\|\tilde{\e}_h^{\ell+2}\|_2^2+\mathcal{C}(k^4+h^4).
		\end{align}
		
		\item Estimate of $\tilde{I}_4$: It follows from \cref{lem27_2} in \cref{lem27} that
		\begin{align}\label{I4}
		\tilde{I}_4 = \,&-\alpha \left\langle \left( 2{\m}_h^{\ell+1}-{\m}_h^\ell\right)\times\left( (2{\m}_h^{\ell+1}-{\m}_h^\ell)\times\Delta_h\tilde{\e}_h^{\ell+2}\right),\tilde{\e}_h^{\ell+2}\right\rangle\\
		= \, & \alpha \left\langle \left( (2{\m}_h^{\ell+1}-{\m}_h^\ell)\times\Delta_h\tilde{\e}_h^{\ell+2}\right)\times\left( 2{\m}_h^{\ell+1}-{\m}_h^\ell\right),\tilde{\e}_h^{\ell+2}\right\rangle\nonumber\\
		= \, & \alpha \left\langle (2{\m}_h^{\ell+1}-{\m}_h^\ell)\times\big[\tilde{\e}_h^{\ell+2}\times (2{\m}_h^{\ell+1}-{\m}_h^\ell)\big],\Delta_h\tilde{\e}_h^{\ell+2} \right\rangle\nonumber\\
		= \, & \alpha \left\langle \nabla_h\Big((2{\m}_h^{\ell+1}-{\m}_h^\ell)\times\big[\tilde{\e}_h^{\ell+2}\times (2{\m}_h^{\ell+1}-{\m}_h^\ell)\big]\Big),\nabla_h\tilde{\e}_h^{\ell+2} \right\rangle \nonumber\\
		\le \, & \mathcal{C} \Big( \|\nabla_h\tilde{\e}_h^{\ell+2}\|_2^2+\|\nabla_h(2{\m}_h^{\ell+1}-{\m}_h^\ell)\|_{\infty}^2\cdot \|\tilde{\e}_h^{\ell+2}\|_2^2\cdot \|2{\m}_h^{\ell+1}-{\m}_h^\ell\|_{\infty}^2\nonumber\\
		&+\|2{\m}_h^{\ell+1}-{\m}_h^\ell\|_{\infty}^2\cdot \|\nabla_h \tilde{\e}_h^{\ell+2} \|_2^2\cdot \|2{\m}_h^{\ell+1}-{\m}_h^\ell\|_{\infty}^2\nonumber\\
		&+\|2{\m}_h^{\ell+1}-{\m}_h^\ell\|_{\infty}^2\cdot \|\tilde{\e}_h^{\ell+2}\|_2^2\cdot \|\nabla_h(2{\m}_h^{\ell+1}-{\m}_h^\ell\|_{\infty}^2)\Big) \nonumber\\
		\le \, & \mathcal{C} ( \| \nabla_h \tilde{\e}_h^{\ell+2}\|_2^2
		+ \|\tilde{\e}_h^{\ell+2}\|_2^2 ) . \nonumber
		\end{align}
		\item Estimates of $\tilde{I}_5$ and $\tilde{I}_6$:
		\begin{align}\label{I5}
		\tilde{I}_5 = \,&-\alpha \left\langle\left( 2{\m}_h^{\ell+1}-{\m}_h^\ell\right)\times\left( (2{\e}_h^{\ell+1}-{\e}_h^\ell)\times\Delta_h\um_h^{\ell+2} \right),\tilde{\e}_h^{\ell+2} \right\rangle\\
		\le \, & \frac{\alpha}{2}\Big( \|\tilde{\e}_h^{\ell+2}\|_2^2+\|2{\m}_h^{\ell+1}-{\m}_h^\ell\|_{\infty}^2\cdot \|2{\e}_h^{\ell+1}-{\e}_h^\ell\|_2^2\cdot \|\Delta_h\um_h^{\ell+2}\|_{\infty}^2 \Big) \nonumber\\
		\le \, & \mathcal{C} ( \|\tilde{\e}_h^{\ell+2}\|_2^2 + \| \e_h^{\ell+1}\|_2^2
		+ \| \e_h^{\ell}\|_2^2 ) . \nonumber
		\end{align}
		
		\begin{align}\label{I6}
		\tilde{I}_6 =\,&-\alpha \left\langle \left( 2{\e}_h^{\ell+1}-{\e}_h^\ell\right)\times\left( (2\um_h^{\ell+1}-\um_h^\ell)\times\Delta_h\um_h^{\ell+2}\right),\tilde{\e}_h^{\ell+2}\right\rangle\\
		\le \, & \frac{\alpha}{2}\Big( \|\tilde{\e}_h^{\ell+2}\|_2^2+\|2{\e}_h^{\ell+1}-{\e}_h^\ell\|_2^2\cdot \|2\um_h^{\ell+1}-\um_h^\ell\|_{\infty}^2\cdot \|\Delta_h\um_h^{\ell+2}\|_{\infty}^2\Big) \nonumber\\
		\le \, & \mathcal{C} ( \|\tilde{\e}_h^{\ell+2}\|_2^2 + \| \e_h^{\ell+1}\|_2^2
		+ \| \e_h^{\ell}\|_2^2 ) . \nonumber
		\end{align}
		
	\end{itemize}	
	Meanwhile, the inner product of the left hand side of \cref{ccc73} with $\tilde{\e}_h^{\ell+2}$ turns out to be	
	\begin{align*}
	L.H.S.&=\frac{1}{4k}\big(\|\tilde{\e}_h^{\ell+2}\|_2^2-\|\tilde{\e}_h^{\ell+1}\|_2^2+\|2\tilde{\e}_h^{\ell+2}-\tilde{\e}_h^{\ell+1}\|_2^2-\|2\tilde{\e}_h^{\ell+1}-\tilde{\e}_h^\ell\|_2^2\\
	&+\|\tilde{\e}_h^{\ell+2}-2\tilde{\e}_h^{\ell+1}+\tilde{\e}_h^\ell\|_2^2\big).\nonumber
	\end{align*}
	Its combination with \cref{I1,I2,I3,I4,I5,I6} and \cref{rhs} leads to
	\begin{align}\label{ccc34}
	&\|\tilde{\e}_h^{\ell+2}\|_2^2-\|\tilde{\e}_h^{\ell+1}\|_2^2+\|2\tilde{\e}_h^{\ell+2}-\tilde{\e}_h^{\ell+1}\|_2^2-\|2\tilde{\e}_h^{\ell+1}-\tilde{\e}_h^\ell\|_2^2\\
	\le \, & \mathcal{C} k ( \|\nabla_h\tilde{\e}_h^{\ell+2}\|_2^2  + \|\tilde{\e}_h^{\ell+2}\|_2^2 + \| \e_h^{\ell+1}\|_2^2 + \| \e_h^{\ell}\|_2^2 ) + \mathcal{C} k (k^4+h^4) 	. \nonumber
	\end{align}
	
	However, the standard $L^2$ error estimate~\cref{ccc34} does not allow one to apply discrete Gronwall inequality, due to the $H_h^1$ norms of the error function involved on the right hand side. To overcome this difficulty, we take a discrete inner product with the numerical error equation~\cref{ccc73} by $-\Delta_h \tilde{\e}_h^{\ell+2}$ and see that
	\begin{align}
	R.H.S.&= \left\langle-\left( 2{\m}_h^{\ell+1}-{\m}_h^\ell\right)\times\Delta_h\tilde{\e}_h^{\ell+2},-\Delta_h\tilde{\e}_h^{\ell+2} \right\rangle\\
	&-\left\langle\left( 2{\e}_h^{\ell+1}-{\e}_h^\ell\right)\times\Delta_h\underline{\m}_h^{\ell+2}, -\Delta_h\tilde{\e}_h^{\ell+2}\right\rangle
	+\left\langle\tau_h^{\ell+2},-\Delta_h\tilde{\e}_h^{\ell+2} \right\rangle \nonumber\\
	&-\alpha \left\langle \left( 2{\m}_h^{\ell+1}-{\m}_h^\ell\right)\times\left( (2{\m}_h^{\ell+1}-{\m}_h^\ell)\times\Delta_h\tilde{\e}_h^{\ell+2}\right),-\Delta_h\tilde{\e}_h^{\ell+2}\right\rangle\nonumber\\
	&-\alpha \left\langle\left( 2{\m}_h^{\ell+1}-{\m}_h^\ell\right)\times\left( (2{\e}_h^{\ell+1}-{\e}_h^\ell)\times\Delta_h\underline{\m}_h^{\ell+2} \right),-\Delta_h\tilde{\e}_h^{\ell+2} \right\rangle\nonumber\\
	&-\alpha \left\langle \left( 2{\e}_h^{\ell+1}-{\e}_h^\ell\right)\times\left( (2\underline{\m}_h^{\ell+1}-\underline{\m}_h^\ell)\times\Delta_h\underline{\m}_h^{\ell+2}\right),-\Delta_h\tilde{\e}_h^{\ell+2}\right\rangle\nonumber\\
	&=: I_1+I_2+I_3+I_4+I_5+I_6.\nonumber
	\end{align}
	\begin{itemize}
		\item Estimate of $I_1$:
		\begin{align}
		I_1 =\left\langle-(2{\m}_h^{\ell+1}-{\m}_h^\ell)\times\Delta_h\tilde{\e}_h^{\ell+2},-\Delta_h\tilde{\e}_h^{\ell+2}\right\rangle=0.\label{ccc11}
		\end{align}
		\item Estimate of $I_2$: 		
		\begin{align}\label{ccc14}
		I_2 =\, &-\left\langle (2{\e}_h^{\ell+1}-{\e}_h^\ell)\times\Delta_h\um_h^{\ell+2},-\Delta_h\tilde{\e}_h^{\ell+2}\right\rangle \\
		= \, & \left\langle \nabla_h\left(\Delta_h\um_h^{\ell+2}\times(2{\e}_h^{\ell+1}-{\e}_h^\ell) \right),\nabla_h\tilde{\e}_h^{\ell+2} \right\rangle \nonumber\\
		\le \, & \mathcal{C}\Big(\|\nabla_h\tilde{\e}_h^{\ell+2}\|_2^2+\|\Delta_h\m_h^{\ell+2}\|_{\infty}^2\cdot\|\nabla_h(2{\e}_h^{\ell+1}-{\e}_h^\ell)\|_2^2\nonumber\\
		&+\|\nabla_h(\Delta_h\m_h^{\ell+2})\|_{\infty}^2\cdot\|2{\e}_h^{\ell+1}-{\e}_h^\ell\|_2^2\Big) \nonumber \\
		\le \, & \mathcal{C} \Big( \|\nabla_h\tilde{\e}_h^{\ell+2}\|_2^2 + \|\nabla_h\e_h^{\ell+1}\|_2^2+\|\nabla_h\e_h^{\ell}\|_2^2 +\|\e_h^{\ell+1}\|_2^2+\|\e_h^{\ell}\|_2^2\Big) .\nonumber
		\end{align}
		
		\item Estimate of the truncation error term $I_3$:
		\begin{align}\label{ccc23}
		I_3=\left\langle-\Delta_h\tilde{\e}_h^{\ell+2},\tau^{\ell+2}\right\rangle
		&\leq \mathcal{C}\|\nabla_h\tilde{\e}_h^{\ell+2}\|_2^2+\mathcal{C}(k^4+h^4).
		\end{align}		
		\item Estimate of $I_4$: It follows from~\cref{lem27_3} in~\cref{lem27} that
		\begin{align}\label{ccc16}
		I_4 = \,&-\alpha \left\langle (2{\m}_h^{\ell+1}-{\m}_h^\ell)\times\big((2{\m}_h^{\ell+1}-{\m}_h^\ell)\times\Delta_h\tilde{\e}_h^{\ell+2}\big),-\Delta_h\tilde{\e}_h^{\ell+2}\right\rangle \\
		= \, & \alpha \left\langle (2{\m}_h^{\ell+1}-{\m}_h^\ell)\times\Delta_h\tilde{\e}_h^{\ell+2} , \Delta_h\tilde{\e}_h^{\ell+2}\times(2{\m}_h^{n+1}-{\m}_h^\ell) \right\rangle \nonumber\\
		= \, & -\alpha \|(2{\m}_h^{\ell+1}-{\m}_h^\ell)\times\Delta_h\tilde{\e}_h^{\ell+2}\|_2^2\leq 0. \nonumber
		\end{align}
		
		\item Estimates of $I_5$ and $I_6$:
		\begin{align}\label{point21}
		I_5=\,&-\alpha \left\langle (2{\m}_h^{\ell+1}-{\m}_h^\ell)\times\big((2{\e}_h^{\ell+1}-{\e}_h^\ell)\times\Delta_h\um_h^{\ell+2}\big),-\Delta_h\tilde{\e}_h^{\ell+2}\right\rangle \\
		= \, & - \alpha \left\langle \nabla_h\left( (2{\m}_h^{\ell+1}-{\m}_h^\ell)\times\big((2{\e}_h^{\ell+1}-{\e}_h^\ell)\times\Delta_h\um_h^{\ell+2}\big)\right),\nabla_h\tilde{\e}_h^{\ell+2} \right\rangle \nonumber\\
		\le \, & \mathcal{C}\Big( \|\nabla_h\tilde{\e}_h^{\ell+2}\|_2^2+\|\nabla_h(2{\m}_h^{\ell+1}-{\m}_h^\ell)\|_{\infty}^2\cdot\|\Delta_h\um_h^{\ell+2}\|_{\infty}^2\cdot\|2{\e}_h^{\ell+1}-{\e}_h^\ell\|_2^2\nonumber\\
		&+\|2{\m}_h^{\ell+1}-{\m}_h^\ell\|_{\infty}^2\cdot\|\nabla_h(\Delta_h\um_h^{\ell+2})\|_{\infty}^2\cdot\|2{\e}_h^{\ell+1}-{\e}_h^\ell\|_2^2\nonumber\\
		&+\|2{\m}_h^{\ell+1}-{\m}_h^\ell\|_{\infty}^2\cdot\|\Delta_h\um_h^{\ell+2}\|_{\infty}^2\cdot\|\nabla_h(2{\e}_h^{\ell+1}-{\e}_h^\ell)\|_2^2\Big) \nonumber\\
		\le \, & \mathcal{C} \Big( \|\nabla_h\tilde{\e}_h^{\ell+2}\|_2^2 + \|\nabla_h\e_h^{\ell+1}\|_2^2+\|\nabla_h\e_h^{\ell}\|_2^2 +\|\e_h^{\ell+1}\|_2^2+\|\e_h^{\ell}\|_2^2\Big) .			\nonumber
		\end{align}
		
		\begin{align}\label{point22}
		I_6 = \,&-\alpha \left\langle (2{\e}_h^{\ell+1}-{\e}_h^\ell)\times\big((2\um_h^{\ell+1}-\um_h^\ell)\times\Delta_h\um_h^{\ell+2}\big),-\Delta_h\tilde{\e}_h^{\ell+2}\right\rangle \\
		= \, & - \alpha \left\langle \nabla_h\Big[(2{\e}_h^{\ell+1}-{\e}_h^\ell)\times\big((2\um_h^{\ell+1}-\um_h^\ell)\times\Delta_h\um_h^{\ell+2}\big)\Big],\nabla_h\tilde{\e}_h^{\ell+2}\right\rangle \nonumber\\
		\le \, & \mathcal{C} \Big( \|\nabla_h\tilde{\e}_h^{\ell+2}\|_2^2+\|\nabla_h(2{\e}_h^{\ell+1}-{\e}_h^\ell)\|_2^2\cdot\|2\um_h^{\ell+1}-\um_h^\ell\|_{\infty}^2\cdot\|\Delta_h\um_h^{\ell+2}\|_{\infty}^2\nonumber\\
		&+\|2{\e}_h^{\ell+1}-{\e}_h^\ell\|_2^2\cdot\|\nabla_h(2\um_h^{\ell+1}-\um_h^\ell)\|_{\infty}^2\cdot\|\Delta_h\um_h^{\ell+2}\|_{\infty}^2\nonumber\\
		&+\|2{\e}_h^{\ell+1}-{\e}_h^\ell\|_2^2\cdot\|2\um_h^{\ell+1}-\um_h^\ell\|_{\infty}^2\cdot\|\nabla_h(\Delta_h\um_h^{\ell+2})\|_{\infty}^2 \Big) \nonumber\\
		\le \, & \mathcal{C} \Big( \|\nabla_h\tilde{\e}_h^{\ell+2}\|_2^2 + \|\nabla_h\e_h^{\ell+1}\|_2^2+\|\nabla_h\e_h^{\ell}\|_2^2 +\|\e_h^{\ell+1}\|_2^2+\|\e_h^{\ell}\|_2^2\Big) .			\nonumber
		\end{align}
		
	\end{itemize}
	And also, the inner product on the left hand side becomes
	\begin{align}\label{ccc25}
	L.H.S.&=\frac{1}{4k}\big( \|\nabla_h\tilde{\e}_h^{\ell+2}\|_2^2-\|\nabla_h\tilde{\e}_h^{\ell+1}\|_2^2+\|2\nabla_h\tilde{\e}_h^{\ell+2}-\nabla_h\tilde{\e}_h^{\ell+1}\|_2^2\\
	&-\|2\nabla_h\tilde{\e}_h^{\ell+1}-\nabla_h\tilde{\e}_h^\ell\|_2^2+\|\nabla_h\tilde{\e}_h^{\ell+2}-2\nabla_h\tilde{\e}_h^{\ell+1}+\nabla_h\tilde{\e}_h^\ell\|_2^2\big).\nonumber
	\end{align}
	Substituting~\cref{ccc11}, \cref{ccc14,ccc16,point21,point22,ccc23} into~\cref{ccc73}, combined with~\cref{ccc25}, we arrive at
	\begin{align}\label{26}
	&\|\nabla_h\tilde{\e}_h^{\ell+2}\|_2^2-\|\nabla_h\tilde{\e}_h^{\ell+1}\|_2^2+\|2\nabla_h\tilde{\e}_h^{\ell+2}-\nabla_h\tilde{\e}_h^{\ell+1}\|_2^2-\|2\nabla_h\tilde{\e}_h^{\ell+1}-\nabla_h\tilde{\e}_h^\ell\|_2^2\\
	\le \, & \mathcal{C} k \Big( \|\nabla_h\tilde{\e}_h^{\ell+2}\|_2^2 + \|\nabla_h\e_h^{\ell+1}\|_2^2+\|\nabla_h\e_h^{\ell}\|_2^2 +\|\e_h^{\ell+1}\|_2^2+\|\e_h^{\ell}\|_2^2\Big) +\mathcal{C}k (k^4+h^4) . \nonumber
	\end{align}
	As a consequence, a combination of~\cref{ccc34} and~\cref{26} yields
	\begin{align}\label{convergence-1}
	&\|\tilde{\e}_h^{\ell+2}\|_2^2-\|\tilde{\e}_h^{\ell+1}\|_2^2+\|2\tilde{\e}_h^{\ell+2}-\tilde{\e}_h^{\ell+1}\|_2^2-\|2\tilde{\e}_h^{\ell+1}-\tilde{\e}_h^\ell\|_2^2 \\	
	&+ \|\nabla_h\tilde{\e}_h^{\ell+2}\|_2^2 - \|\nabla_h\tilde{\e}_h^{\ell+1}\|_2^2
	+\| \nabla_h ( 2\tilde{\e}_h^{\ell+2} - \tilde{\e}_h^{\ell+1} ) \|_2^2
	- \| \nabla_h (2 \tilde{\e}_h^{\ell+1} - \tilde{\e}_h^\ell\ ) \|_2^2 \nonumber \\
	\le \, & \mathcal{C} k \Big( \|\nabla_h\tilde{\e}_h^{\ell+2}\|_2^2
	+ \| \tilde{\e}_h^{\ell+2}\|_2^2 + \|\nabla_h\e_h^{\ell+1}\|_2^2
	+ \|\nabla_h\e_h^{\ell}\|_2^2 +\|\e_h^{\ell+1}\|_2^2+\|\e_h^{\ell}\|_2^2 \Big)
	\nonumber  \\
	&  +\mathcal{C}k (k^4+h^4) . \nonumber
	\end{align}
	At this point, recalling the $W_h^{1,\infty}$ bound for $\m_h^k$ and $\tilde{\m}_h^k$, as given by~\cref{bound-3}, \cref{bound-4}, and applying~\cref{lem 6-2} in~\cref{lem 6-0}, we obtain
	\begin{equation*}
	\| \e_h^k \|_2 \le 2 \| \tilde{\e}_h^k \|_2 + \mathcal{O} (h^2) ,  \, \, \,
	\| \nabla_h \e_h^k \|_2 \le \mathcal{C} ( \| \nabla_h \tilde{\e}_h^k \|_2
	+ \| \tilde{\e}_h^k \|_2 ) + \mathcal{O} (h^2) , \quad \mbox{$k = \ell, \ell+1$} .  \label{convergence-2}
	\end{equation*}
	Its substitution into~\cref{convergence-1} leads to
	\begin{align*}
	&\|\tilde{\e}_h^{\ell+2}\|_2^2-\|\tilde{\e}_h^{\ell+1}\|_2^2+\|2\tilde{\e}_h^{\ell+2}-\tilde{\e}_h^{\ell+1}\|_2^2-\|2\tilde{\e}_h^{\ell+1}-\tilde{\e}_h^\ell\|_2^2 \\	
	&+ \|\nabla_h\tilde{\e}_h^{\ell+2}\|_2^2 - \|\nabla_h\tilde{\e}_h^{\ell+1}\|_2^2
	+\| \nabla_h ( 2\tilde{\e}_h^{\ell+2} - \tilde{\e}_h^{\ell+1} ) \|_2^2
	- \| \nabla_h (2 \tilde{\e}_h^{\ell+1} - \tilde{\e}_h^\ell\ ) \|_2^2 \nonumber \\
	\le \, & \mathcal{C} k \Big( \|\nabla_h\tilde{\e}_h^{\ell+2}\|_2^2
	+ \|\nabla_h \tilde{\e}_h^{\ell+1}\|_2^2
	+ \|\nabla_h \tilde{\e}_h^{\ell} \|_2^2 + \| \tilde{\e}_h^{\ell+2}\|_2^2
	+\| \tilde{\e}_h^{\ell+1}\|_2^2 + \| \tilde{\e}_h^{\ell}\|_2^2 \Big)
	\nonumber  \\
	&  +\mathcal{C}k (k^4+h^4) . \nonumber
	\end{align*}
	In turn, an application of discrete Gronwall inequality (in~\cref{ccclem1}) yields the desired convergence estimate for $\tilde{\e}_h$:
	\begin{align*}
	\| \tilde{\e}_h^n\|_2^2 + \|\nabla_h\tilde{\e}_h^n\|_2^2\leq \mathcal{C}Te^{\mathcal{C}T}(k^4+h^4), \quad \text{for all } n:n\leq \left\lfloor\frac{T}{k}\right\rfloor , 
	\end{align*}
	i.e.,
	\begin{align*} 
	\|\tilde{\e}_h^n\|_2 + \|\nabla_h\tilde{\e}_h^n\|_2
	&\le \mathcal{C}(k^2+h^2) .
	\end{align*}
	An application of~\cref{ccclemC1}, as well as the time step constraint $k\leq \mathcal{C}h$, leads to
	\begin{align}\label{boundd-5}
	\|\tilde{\e}_h^n\|_{\infty} &\leq \frac{\|\tilde{\e}_h^n\|_2}{h^{d/2}}\leq \frac{\mathcal{C}(k^2+h^2)}{h^{d/2}}\leq \frac{1}{6}, \\
	\|\nabla_h\tilde{\e}_h^n\|_{\infty} &\leq \frac{\|\nabla_h\tilde{\e}_h^n\|_2}{h^{d/2}}\leq \frac{\mathcal{C}(k^2+h^2)}{h^{d/2}}\leq \frac{1}{6},\nonumber
	\end{align}
	so that the second part of the \textit{a priori} assumption~\cref{bound-2} has been recovered at time step $k=n$. In turn, the $W_h^{1,\infty}$ bound~\cref{bound-4} becomes available, which enables us to apply~\cref{lem 6-2} in~\cref{lem 6-0}, and obtain the desired convergence estimate for $\e_h^n$:
	\begin{align*}
	& \| \e_h^n \|_2 \le 2 \| \tilde{\e}_h^n \|_2 + \mathcal{O} (h^2)
	\le \mathcal{C} (k^2 + h^2)  ,  \label{convergence-6}  \\
	& \| \nabla_h \e_h^n \|_2 \le \mathcal{C} ( \| \nabla_h \tilde{\e}_h^n \|_2
	+ \| \tilde{\e}_h^n \|_2 ) + \mathcal{O} (h^2)
	\le \mathcal{C} (k^2 + h^2) .  \nonumber
	\end{align*}
	Similar to the derivation of~\cref{boundd-5}, we also get
	\begin{align*}  
	\| \e_h^n\|_{\infty} &\leq  \frac{1}{6} , \quad
	\|\nabla_h\e_h^n\|_{\infty} \leq \frac{1}{6} ,
	\end{align*}
	so that the first part of the \textit{a priori} assumption~\cref{bound-2} has been recovered at time step $k=n$. This completes the proof of~\cref{cccthm2}.
\end{proof}

\section{Numerical examples}
\label{sec:experiments}

In this section, we perform 1-D and 3-D numerical experiments for the final time $T=1$ to verify the theoretical analysis in \cref{sec:main theory}.
Rate of convergence is obtained via the least-squares fitting for a sequence of error data recorded with successive step-size refinements.

In details, we test four examples: 1-D example with a forcing term and the given exact solution, 1-D example without the exact solution,
and 3-D example with a forcing term and the given exact solution, 3-D example with respect to the domain wall dynamics without exact solution for full Landau-Lifshitz equation in \cite{Xie2018}. Solutions in these four cases satisfy the homogenous Neumann boundary
condition \cref{boundary}. In the presence of an forcing term, the LL equation reads as
\begin{equation*}
{\m}_t=-{\m}\times\Delta{\m}-\alpha{\m}\times({\m}\times\Delta{\m})+\f , \label{cccc4}
\end{equation*}
with $\f={\m}_{et}+{\m}_e\times\Delta{\m_e}+\alpha{\m_e}\times\left( {\m_e}\times\Delta{\m_e}\right)$
and $\m_e$ the exact solution. In more details, the forcing term $\f$ is evaluated at $t^{n+2}$ in the numerical scheme \cref{scheme-1-1}. Only one linear system of equations needs to solve at each time step. In all examples, we find that the
scheme is unconditionally stable.

\begin{example}[1-D example with the given exact solution]
	The given exact solution is $\m_e=\left(\cos(x^2(1-x)^2)\sin t, \sin(x^2(1-x)^2)\sin t, \cos t\right)^T$,
	which satisfies the homogeneous Neumman boundary condition.
	Results in \cref{ccctab-1} and \cref{BDF2_plus_LL} suggest the second-order accuracy in both time and space
	of the proposed method in the discrete $H^1-$norm; and in \cref{cccctab-1} indicate the unconditional stability of our method in the 1D case. 
	
	\begin{table}[htbp]
		\caption{Accuracy of our method on the uniform mesh when $h=k$ and $\alpha=0.01$. }
		\label{ccctab-1}
		\centering
		\begin{tabular}{c|c|c|c}
			\hline \hline
			$k$ & $\|\m_h-\m_e\|_{\infty}$ & $\|\m_h-\m_e\|_{2}$ & $\|\m_h-\m_e\|_{H^1}$ \\
			\hline
			5.0D-3 & 3.867D-5 & 4.115D-5 & 1.729D-4 \\
			2.5D-3 & 7.976D-6 & 1.053D-5 & 4.629D-5 \\
			1.25D-3 & 2.135D-6 & 2.648D-6 & 1.177D-5 \\
			6.25D-4 & 5.765D-7 & 6.627D-7 & 2.949D-6 \\
			3.125D-4 & 1.447D-7 & 1.657D-7 & 7.370D-7 \\
			\hline
			order & 1.991 & 1.990 & 1.972 \\
			\hline \hline
		\end{tabular}
	\end{table}
	
	\begin{table}[!htbp]
		\caption{No stability constraint of $k$ for our method in 1D case when $\alpha=0.01$. }
		\label{cccctab-1}
		\centering
		\begin{tabular}{c|c|c|c|c}
			\hline \hline
			\diagbox{$k$}{$\|\m_h-\m_e\|_{\infty}$}{$h$}&1.0D-1&5.0D-2&2.5D-2 &1.25D-2\\
			\hline
			2.0D-1& 2.318D-2 & 2.106D-2 & 2.056D-2 & 2.046D-2 \\
			1.0D-1& 1.015D-2 & 7.571D-3 & 6.928D-3 & 6.768D-3 \\
			5.0D-2& 5.503D-3 & 2.807D-3 & 2.134D-3 & 1.966D-3 \\
			2.5D-2& 4.166D-3 & 1.436D-3 & 7.521D-4 & 5.811D-4 \\
			1.25D-2& 3.783D-3& 1.062D-3 & 3.913D-4 & 2.234D-4 \\
			6.25D-3& 3.709D-3 & 9.714D-4 & 2.831D-4 & 1.108D-4 \\
			\hline \hline
		\end{tabular}
	\end{table}
	
	\begin{figure}[htbp]
		\centering
		\includegraphics[width=0.5\textwidth]{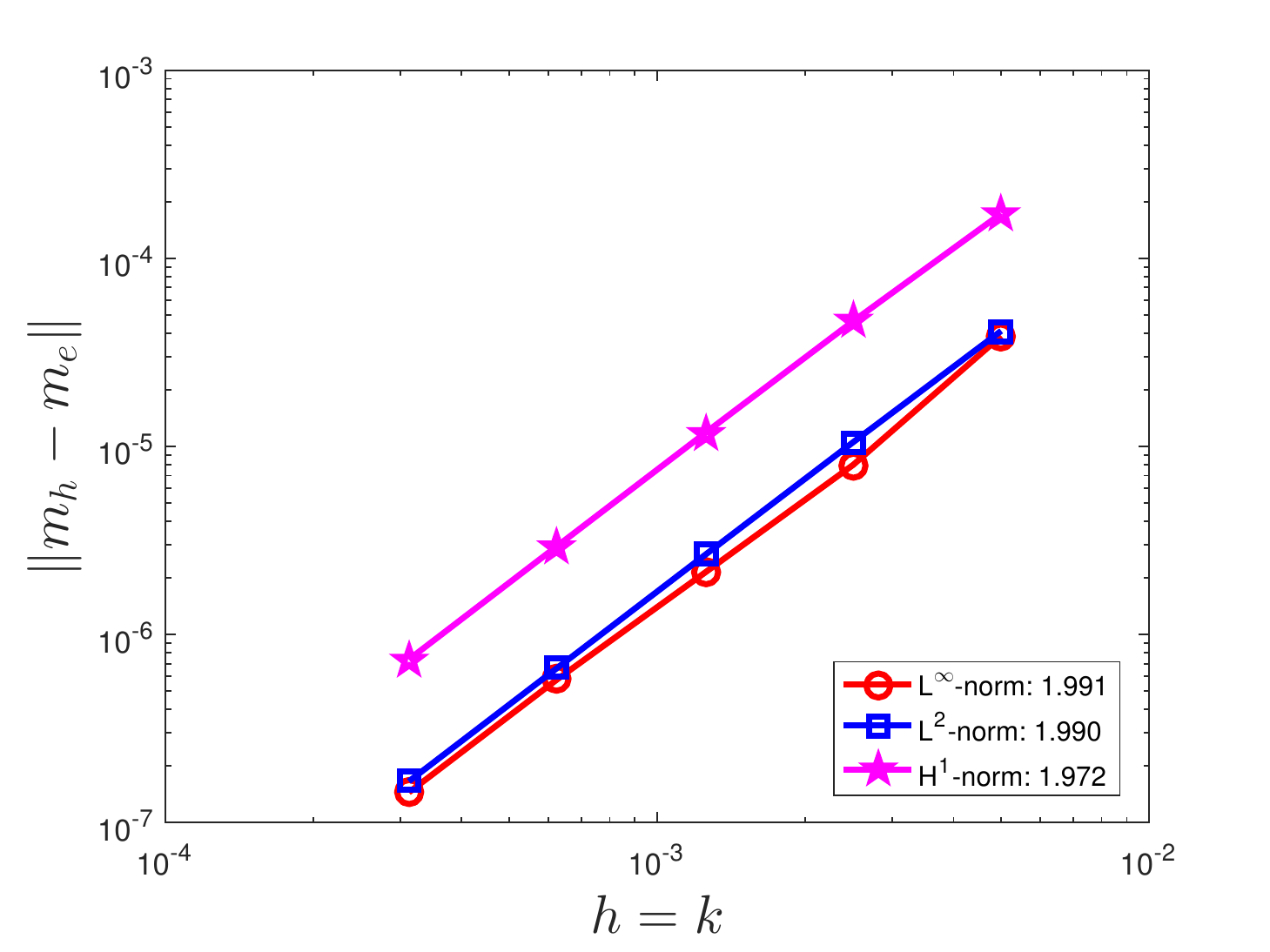}\\
		\caption{Accuracy of our method on the uniform mesh when $h=k$ and $\alpha=0.01$.}\label{BDF2_plus_LL}
	\end{figure}
\end{example}

\begin{example}[1-D example without the exact solution]
	For this example, in the absence of the forcing term, we do not have the exact solution. For comparison, we first set $h$ and $k$
	small enough to obtain a numerical solution which will be used as the exact (reference) solution. In this test, we take the initial condition as $\m_0(\x,0)=(0,0,1)^{T}$ for $x\in \Omega$. To get the temporal accuracy,
	we set $h=1D-4$ and $k=1D-4$ to get the exact solution and then record the temporal error with varying $k$ in \cref{ccctab-3} and \cref{BDF2_plus_prime_temproal_1D}. To get the spatial accuracy, we set $h=1/3^8$ and $k=1D-4$ to get the exact solution and
	record the error in \cref{ccctab-6} and \cref{BDF2_plus_prime_spatial_1D}.  Again, the second-order accuracy in both time and space
	in the discrete $H^1-$norm are confirmed.
	\begin{table}[htbp]
		\caption{Temporal accuracy of our method on the uniform mesh when $h=1D-4$ and $\alpha=0.01$. The exact solution
			is obtained with $h=1D-4$ and $k=1D-4$.}
		\label{ccctab-3}
		\centering
		\begin{tabular}{c|c|c|c}
			\hline \hline
			$k$ & $\|\m_h-\m_e\|_{\infty}$ &$\|\m_h-\m_e\|_{2}$ & $\|\m_h-\m_e\|_{H^1}$ \\
			\hline
			5.0D-3& 2.949D-5& 3.250D-5 & 1.633D-4\\
			2.5D-3& 8.116D-6& 8.429D-6 & 4.393D-5\\
			1.25D-3& 2.125D-6 & 2.114D-6 & 1.118D-5\\
			6.25D-4& 4.851D-7 &5.190D-7 & 2.791D-6\\
			3.125D-4& 1.129D-7 &1.196D-7 & 6.875D-7\\
			\hline
			order& 2.012 & 2.019 & 1.976\\
			\hline \hline
		\end{tabular}	
	\end{table}
	
	\begin{table}[htbp]
		\caption{Spatial accuracy of our method on the uniform mesh when $k=1D-4$ and $\alpha=0.01$.
			The exact solution is obtained with $h=1/3^8$ and $k=1D-4$.}
		\label{ccctab-6}
		\centering
		\begin{tabular}{c|c|c|c}
			\hline \hline
			$h$ &$\|\m_h-\m_e\|_{\infty}$&$\|\m_h-\m_e\|_{2}$ & $\|\m_h-\m_e\|_{H^1}$ \\
			\hline
			$1/3^2$&0.00546 &0.00577 &0.01336 \\
			$1/3^3$&6.101D-4 & 6.430D-4&0.00160 \\
			$1/3^4$&6.782D-5 &7.146D-5 &1.820D-4 \\
			$1/3^5$&7.527D-6 &7.930D-6 & 2.036D-5\\
			$1/3^6$&8.271D-7 &8.714D-7 &2.243D-6 \\
			\hline
			order& 2.001 & 2.002& 1.980\\
			\hline \hline
		\end{tabular}
	\end{table}
	
	\begin{figure}[htbp]
		\centering
		\subfloat[Temporal accuracy]{\label{BDF2_plus_prime_temproal_1D}\includegraphics[width=2.5in]{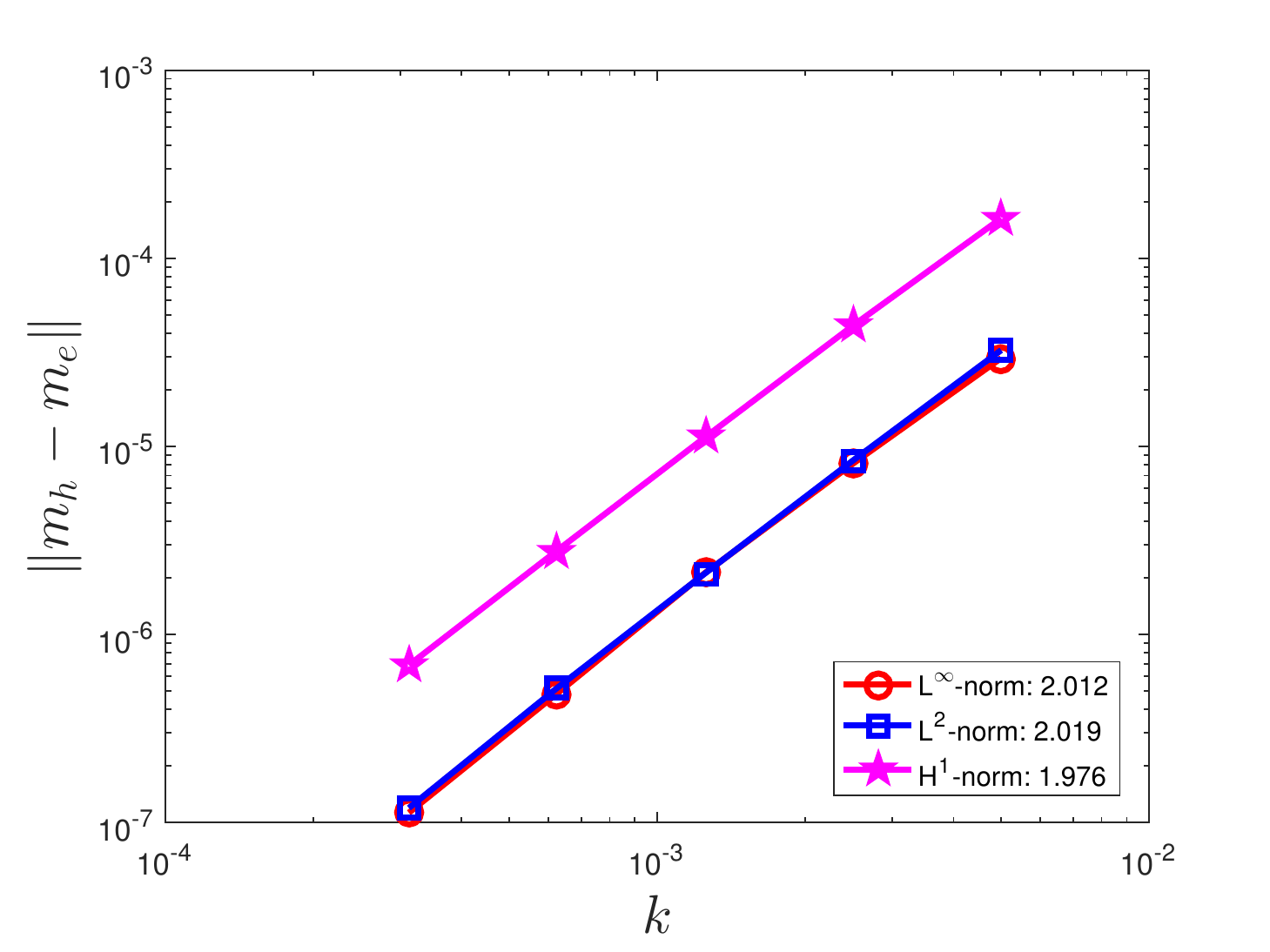}}
		\subfloat[Spatial accuracy]{\label{BDF2_plus_prime_spatial_1D}\includegraphics[width=2.5in]{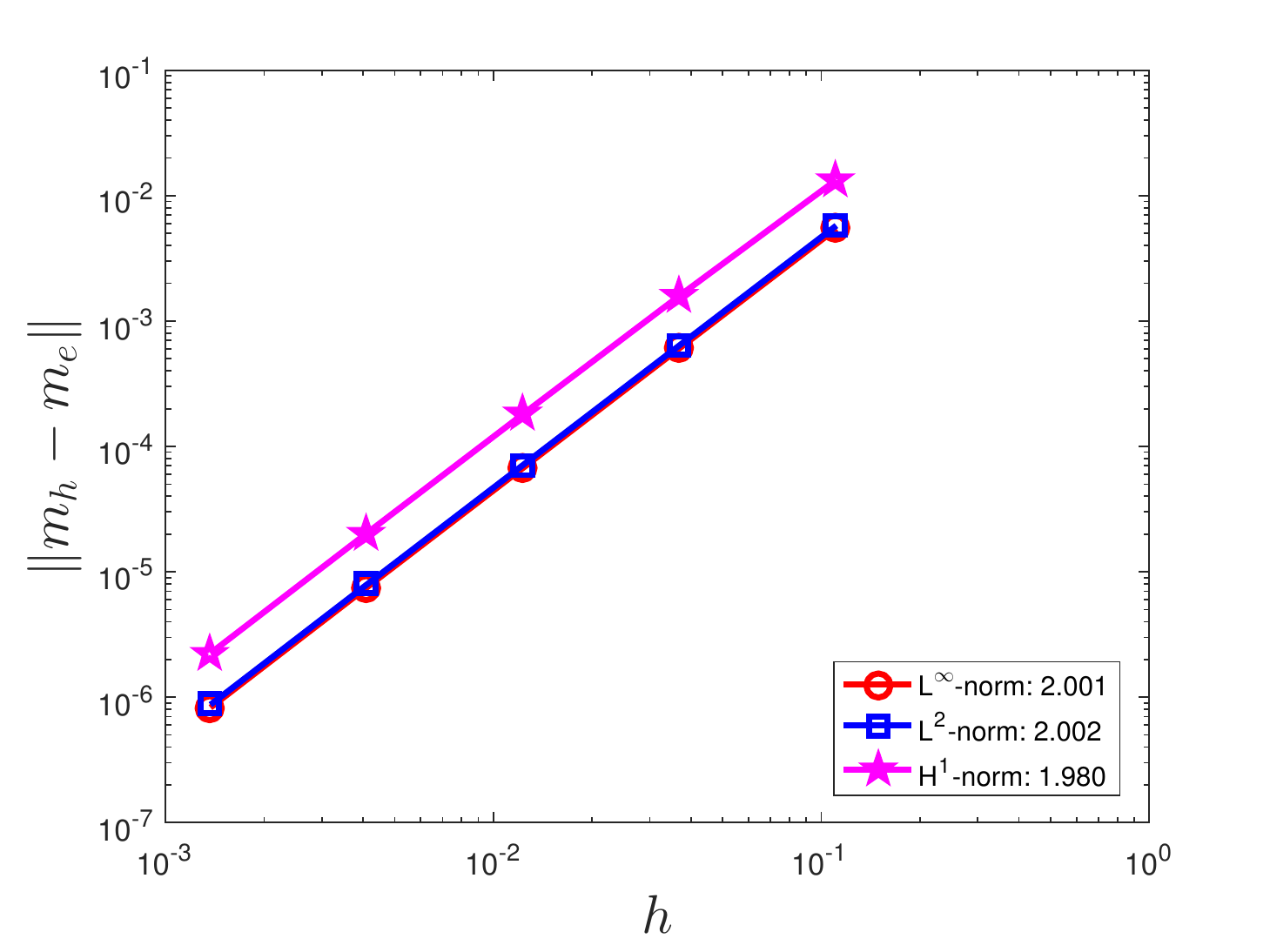}}
		\caption{Accuracy of our method when $\alpha=0.01$. (a) Temporal accuracy of our method on the uniform mesh when $h=1D-4$ and $\alpha=0.01$.
			The exact solution is obtained with $h=1D-4$ and $k=1D-4$; (b) Spatial accuracy of our method on the uniform mesh when $k=1D-4$ and $\alpha=0.01$.
			The exact solution is obtained with $h=1/3^8$ and $k=1D-4$.}
		\label{BDF2_plus_prime_1D}
	\end{figure}

\end{example}

\begin{example}[3-D example with the given exact solution]
	The given exact solution read as
	\begin{equation*}
	\m_e=\left(\cos(XYZ)\sin t, \sin(XYZ)\sin t, \cos t\right)^T,
	\end{equation*}
	where $X=x^2(1-x)^2$, $Y=y^2(1-y)^2$, $Z=z^2(1-z)^2$.
	
	\cref{ccctab-5} shows the second-order convergence in time in the 3-D case. Result in \cref{cccctab-1-3d} indicates the unconditional stability of our method in the 3D case. We visualize the magnetization in~\cref{fig:2D_s}
	by taking a slice along the $z=1/2$ plane. The arrow denotes the vector from magnetization component $u$ to $v$ and the colormap
	represents the third magnetization component $w$. \cref{2D_exact} and \cref{2D_s4} plot the exact magnetization and the numerical magnetization
	when $k=1/256$ and $h_x=h_y=h_z=1/32$, respectively.
	
	\begin{table}[htbp]
		\centering
		\caption{Temporal accuracy in the 3-D case when $h_x=h_y=h_z=1/32$ and $\alpha=0.01$. }\label{ccctab-5}
		\begin{tabular}{c|c|c|c}
			\hline \hline
			$k$&$\|\m_h-\m_e\|_{\infty}$ & $\|\m_h-\m_e\|_2$& $\|\m_h-\m_e\|_{H^1}$ \\
			\hline
			1/16 & 1.685D-3 & 1.098D-3 & 1.211D-3 \\
			1/32 & 4.411D-4 & 2.964D-4 & 3.082D-4 \\
			1/64 & 1.128D-4 & 7.730D-5 & 7.772D-5 \\
			1/128 & 2.966D-5 & 2.024D-5 & 2.051D-5 \\
			1/256 & 8.311D-6 & 5.693D-6 & 5.812D-6 \\
			\hline
			order& 1.922 & 1.906 & 1.932\\
			\hline \hline
		\end{tabular}
	\end{table}
	
	\begin{table}[!htbp]
		\caption{No stability constraint of $k$ for our method in 3D case when $\alpha=0.01$. }
		\label{cccctab-1-3d}
		\centering
		\begin{tabular}{c|c|c|c|c}
			\hline \hline
			\diagbox{$k$}{$\|\m_h-\m_e\|_{\infty}$}{$h$}&1/4&1/8&1/16 &1/32\\
			\hline
			1/4& 1.370D-2 & 1.365D-2 & 1.370D-2 & 1.421D-2 \\
			1/8& 5.470D-3 & 5.415D-3 & 5.407D-3 & 5.686D-3 \\
			1/16& 1.675D-3 & 1.619D-3 & 1.605D-3 & 1.685D-3 \\
			1/32& 5.052D-4 & 4.495D-4 & 4.355D-4 & 4.411D-4 \\
			1/64& 1.860D-4 & 1.303D-4 & 1.163D-4 & 1.128D-4 \\
			1/128& 1.029D-4 & 4.680D-5 & 3.311D-5 & 2.966D-5 \\
			\hline \hline
		\end{tabular}
	\end{table}
	
	\begin{figure}[htbp]
		\centering
		\subfloat[Exact magnetization profile]{\label{2D_exact}\includegraphics[width=2.5in]{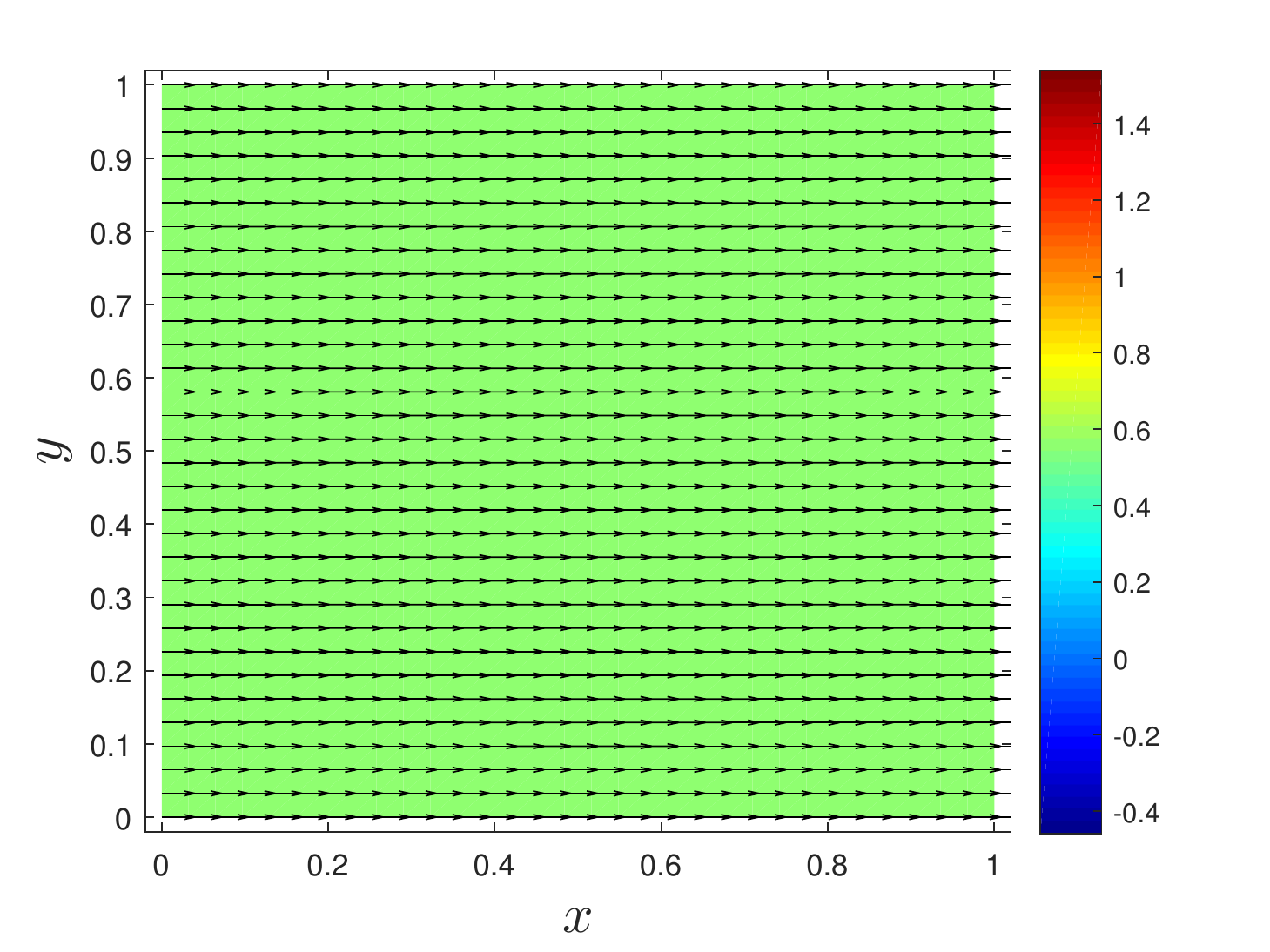}}
		\subfloat[Numerical magnetization profile]{\label{2D_s4}\includegraphics[width=2.5in]{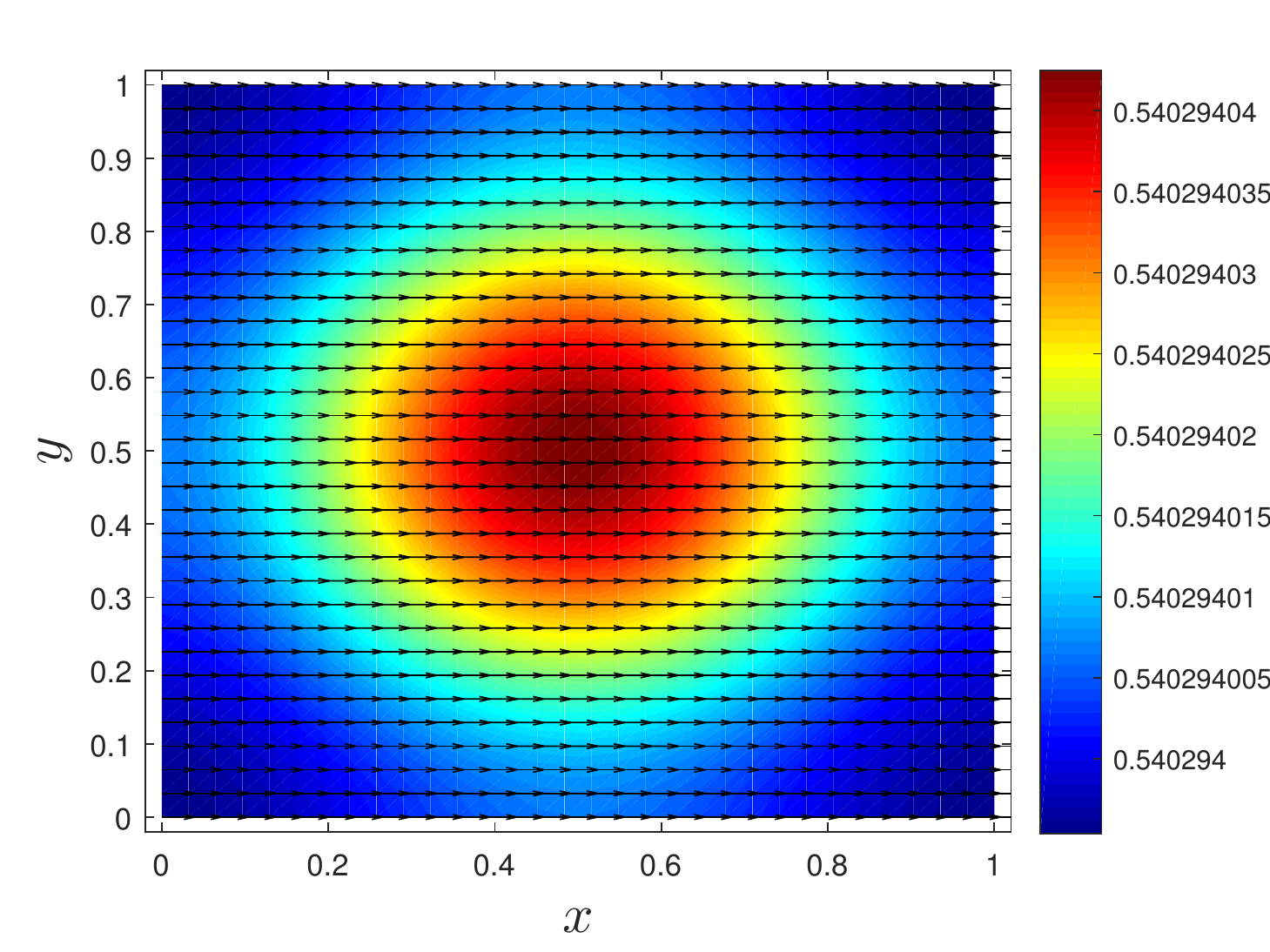}}
		\caption{Profiles of the exact and the numerical magnetization in the $xy-$plane with $z=1/2$ when $k=1/256$, $h_x=h_y=h_z=1/32$, and $\alpha=0.01$.}
		\label{fig:2D_s}
	\end{figure}
	
\end{example}

\begin{example}[3-D example for full Landau-lifshitz equation]
		We consider a magnetic nano strip of size $0.8 \times 0.1 \times 0.004\;\mu m^3$ and of grid points chosen as $128\times 32 \times 2$ in $x$, $y$, $z$ directions respectively. In our simulations, the damping coefficient $\alpha=0.1$ and the time scale is $k=1\;\textrm{ps}$. A stopping criterion is used to determine that a steady state is reached when the relative change in the total energy is less than $10^{-7}$. The transverse domain walls in a magnetic strip are able to be formed by an in plane head-to-head N\'{e}el wall as illustrated in \cref{Neel_wall_initial}. The domain wall dynamics is driven by a small external field imposed of strength $H_e=50\;\textrm{Oe}$. The domain wall moves along $x$ directions with a constant velocity $255\;\textrm{m/s}$. During the motion, the domain wall profile is maintained. The snapshots at time $t=0.5\; \textrm{ns},\; 1.0\; \textrm{ns}$ are shown in \cref{Neel_wall_mdata25,Neel_wall_mdata50}.
		
		\begin{figure}[htbp]
			\centering
			\subfloat[Initial state]{\label{Neel_wall_initial}\includegraphics[width=2.5in]{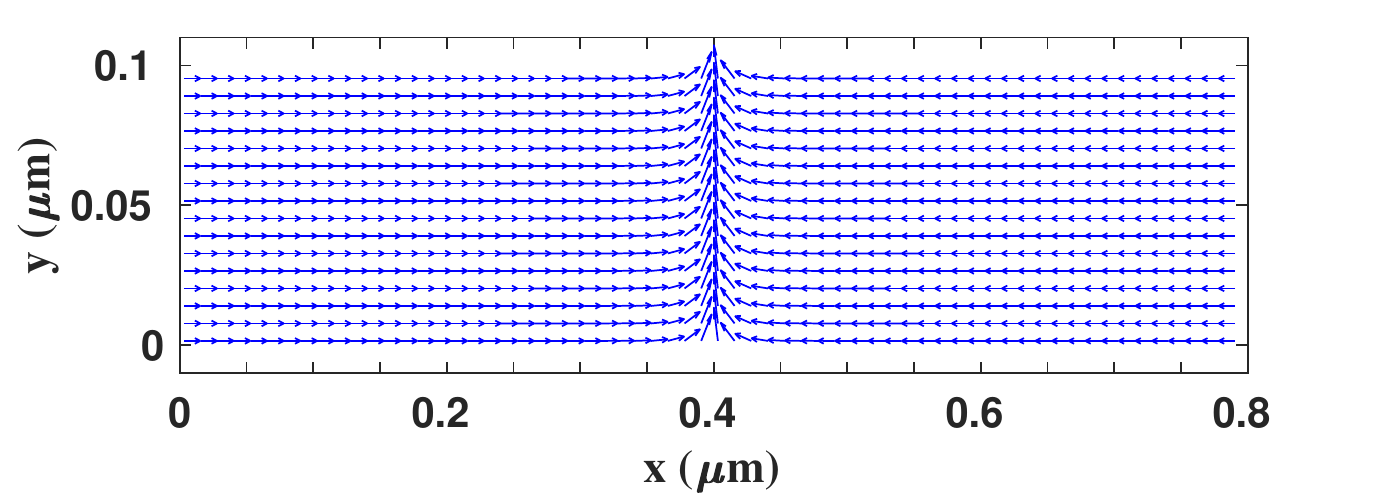}}
			\quad
			\subfloat[Magnetization profile at $t=0.5\; \textrm{ns}$]{\label{Neel_wall_mdata25}\includegraphics[width=2.5in]{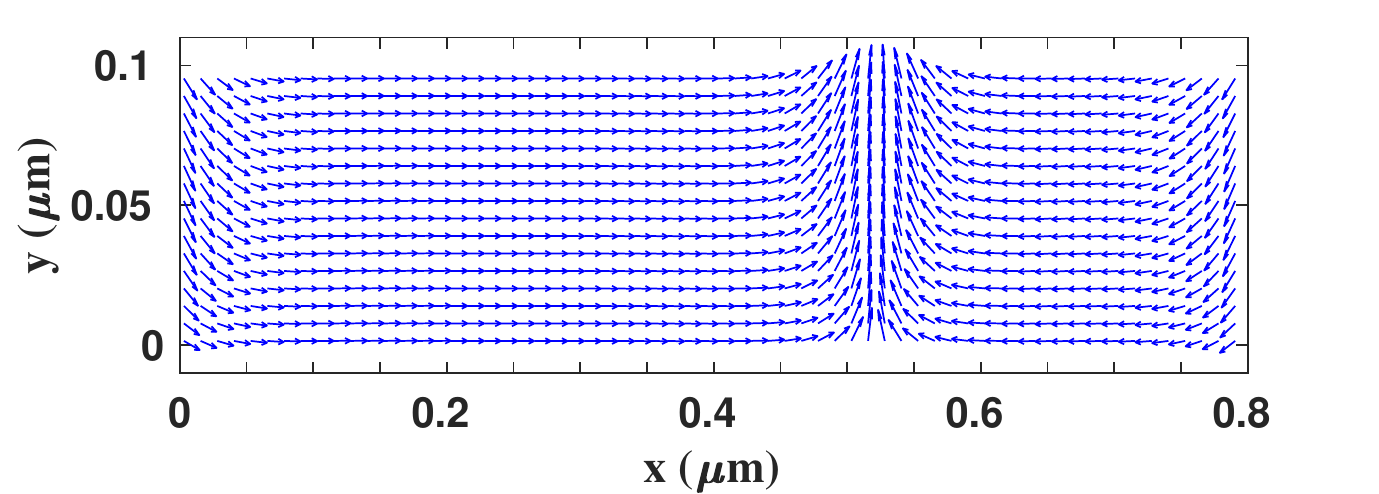}}
			\quad
			\subfloat[Magnetization profile at $t=1.0\; \textrm{ns}$]{\label{Neel_wall_mdata50}\includegraphics[width=2.5in]{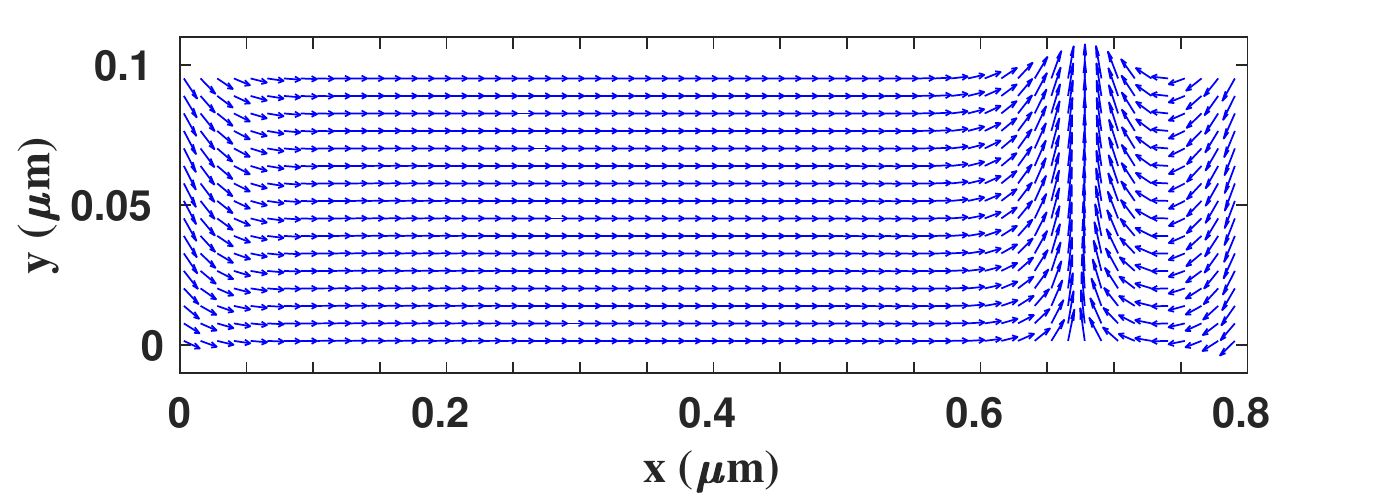}}
			\caption{Snapshots of the domain wall motion for the centered slice in $xy-$plane of the strip with the magnetic field $H_e=50\; \textrm{Oe}$ and the damping constant $\alpha=0.1$ at several times $t=0.0\;\textrm{ns},\; 0.5\; \textrm{ns},\; 1.0\; \textrm{ns}$ in (a),(b) and (c).}
			\label{fig:Neel_wall}
		\end{figure}	
		
	\end{example}

\section{Conclusions}
\label{sec:conclusions}

In this paper, we have proposed and analyzed a second-order time stepping scheme to solve the LL equation.
The second-order BDF is applied for temporal discretization and a linearized multistep approximation is used for
the nonlinear coefficients on the right hand side of the equation. The resulting scheme avoids a well-known
difficulty associated with the nonlinearity of the system, and its unique solvability is established via
the monotonicity analysis of the system. In addition, an optimal rate convergence analysis is provided,
by making use of a linearized stability analysis for the numerical error functions, in which the
$W_h^{1,\infty}$ error estimate at the projection step has played an important role.
Numerical experiments in both 1D and 3D cases are presented to verify the unconditional stability and the second-order convergence
	in both space and time, and applied to the domain wall dynamics driven by the external field. The technique presented here may be applicable to the model for current-driven
domain wall dynamics \cite{ChenGarciaCerveraYang:2015}, which shall be explored as a future project.

\appendix

\section*{Acknowledgments}
We thank Zhennan Zhou from Peking University for helpful discussions.
This work is supported in part by the grants NSFC 21602149, the Young Thousand Talents Program of China, 
and the Innovation and entrepreneurial talent program in Jiangsu (J.~Chen), NSF DMS-1418689 (C.~Wang), and the Innovation Program for postgraduates in Jiangsu province via grant KYCX19\_1947 (C.~Xie).

\bibliographystyle{amsplain}
\bibliography{references}
%
%
%
%

\end{document}